\documentclass[11pt]{article}
\usepackage{amsmath}
\usepackage{amsthm}
\usepackage{amsfonts}
\usepackage{bm}
\usepackage{color}

\usepackage [english]{babel}
\usepackage{comment}

\newtheorem{theo}{Theorem}[section]
\newtheorem{coro}{Corollary}[theo]
\newtheorem{remark}{Remark}[section]

\newcommand{\dsp}{\displaystyle}

\newcommand{\T}{\mathbb{T}}
\newcommand{\Pset}{\mathbb{P}}
\newcommand{\Fi}{u}

\newcommand{\xt}{x}
\newcommand{\ys}{y}
\newcommand{\Fii}{\Phi}
\newcommand{\Muii}{\mu}

\renewcommand{\Re}{\mathop{\rm Re}}
\renewcommand{\Im}{\mathop{\rm Im}}
\newcommand{\const}{\mathop{\rm const}}

\def\Xint#1{\mathchoice
   {\XXint\displaystyle\textstyle{#1}}%
   {\XXint\textstyle\scriptstyle{#1}}%
   {\XXint\scriptstyle\scriptscriptstyle{#1}}%
   {\XXint\scriptscriptstyle\scriptscriptstyle{#1}}%
   \!\int}
\def\XXint#1#2#3{{\setbox0=\hbox{$#1{#2#3}{\int}$}
     \vcenter{\hbox{$#2#3$}}\kern-.5\wd0}}

\def\dashint{\Xint-}

\begin{document}
\setlength{\parskip}{0.0cm}

\title{{$R_{II}$ type recurrence, generalized eigenvalue problem and orthogonal polynomials on the unit circle}\thanks{The second author's  work was supported by funds from CNPq, Brazil (grants 475502/2013-2 and 305073/2014-1). }}
\author{{M.E.H. Ismail$^{1}$ and A. Sri Ranga$^{2}$} \\[1.5ex]
{\small $^{1}$Department of Mathematics, University of Central Florida} \\
{\small Orlando, FL 32816, USA} \\[1ex]
{\small $^{2}$DMAp, IBILCE, UNESP - Universidade Estadual Paulista} \\
{\small 15054-000  S\~{a}o Jos\'{e} do Rio Preto,  SP,  Brazil} \\
{\hspace*{3cm}}\\ }
\date{\today}
\maketitle \thispagestyle{empty}


\begin{abstract} 

We consider a sequence of polynomials $\{P_n\}_{n \geq 0}$ satisfying a special $R_{II}$ type recurrence relation  where the zeros of $P_n$ are simple and lie on the real line. It turns out  that the polynomial $P_n$, for any $n \geq 2$, is the characteristic polynomial of a simple $n \times n$ generalized eigenvalue problem. It is shown that with this $R_{II}$ type recurrence relation one can always associate a positive measure on the unit circle.   The orthogonality property satisfied by $P_n$ with respect to this measure is also obtained.  Finally, examples are given to justify the results.       
\end{abstract}


\vspace{1ex}

Mathematics Subject classification. {15A18, 58C40, 42C05} \\[-1ex]

Keywords. {Generalized eigenvalue problem, Orthogonal polynomials on the unit circle.}

\setcounter{equation}{0}
\section{Introduction}

Recurrence relations of the form 
\[
    P_{n+1}(z) = \sigma_{n+1} (z -v_{n+1}) P_{n}(z) - u_{n+1}(z-a_{n})(z-b_{n}) P_{n-1}(z),  \quad n \geq 1, 
\]
with initial conditions $P_0(z) =  1$ and $P_1(z) = \sigma_1(z - v_1)$, have been studied by Ismail and Masson \cite{Ismail-Masson-JAT1995}. In \cite{Ismail-Masson-JAT1995}  these recurrence relations  were referred to as those associated with  continued fractions of $R_{II}$ type. One of the interesting results shown in \cite{Ismail-Masson-JAT1995} is that given such a recurrence relation, if $u_{n+1} \neq 0$, $P_{n}(a_{k}) \neq 0$ and $P_{n}(b_{k}) \neq 0$, for all $n,k$, then associated with this recurrence relation  there exists a linear functional $\mathcal{L}$ such that the  rational functions $\mathcal{S}_n(z) = P_n(z)  \prod_{j =1}^{n}[(z-a_j)^{-1}(z-b_j)^{-1}]$,  $n \geq 1$,
satisfy the orthogonality 
\begin{equation} \label{Eq-IsmailMasson-Orthogonality}
    \mathcal{L}[z^{j}\mathcal{S}_n(z)] = 0 \quad \mbox{for} \quad  0 \leq j < n. 
\end{equation}

The importance of these recurrence relations were further highlighted in the work of Zhedanov \cite{Zhedanov-JAT1999},  where the author shows that they are connected to generalized eigenvalue problems involving two tri-diagonal matrices.  Just as in  \cite{Zhedanov-JAT1999}  we will refer to these recurrence relations  as $R_{II}$ type recurrence relations. 

Our objective in the present manuscript is to consider the special $R_{II}$ type recurrence relation 
\begin{equation} \label{Eq-Special-R2Type-RR}
     P_{n+1}(\xt) = (\xt - c_{n+1}) P_{n}(\xt)  - d_{n+1} (\xt^2 + 1) P_{n-1}(\xt), \quad n \geq 1,
\end{equation}
with $P_0(\xt) = 1$ and $P_{1}(\xt) = \xt - c_1$, where $\{c_n\}_{n \geq 1}$ is a real sequence and $\{d_{n+1}\}_{n \geq 1}$  is a positive chain sequence.   

We denote by $\{\ell_{n+1}\}_{n \geq 0}$ the minimal parameter sequence of the positive chain sequence  $\{d_{n+1}\}_{n \geq 1}$.  Thus, $(1-\ell_n)\ell_{n+1} = d_{n+1}$, $n \geq 1$, with $\ell_1 = 0$ and $0 < \ell_n < 1$ for $n \geq 2$. 

It is known that a positive chain sequence can have multiple (infinitely many) parameter sequences or just a single parameter sequence.  By Wall's criteria the positive chain sequence $\{d_{n+1}\}_{n \geq 1}$ has only a single parameter sequence if and only if the series 
\[
    \sum_{n=1}^{\infty} \prod_{j=1}^{n} \frac{\ell_{j+1}}{(1-\ell_{j+1})}
\]
is divergent. When the positive chain sequence $\{d_{n+1}\}_{n \geq 1}$ has multiple parameter sequences one could talk about its  maximal parameter sequence $\{M_{n+1}\}_{n \geq 0}$ such that 
\[0 < g_{n} < M_{n} < 1 \quad \mbox{and} \quad  (1-g_{n})g_{n+1} = (1-M_{n})M_{n+1} = d_{n+1}, \quad \mbox{for} \quad  n \geq 1, 
\]
where $\{g_{n+1}\}_{n \geq 0}$ is any other parameter sequence of $\{d_{n+1}\}_{n \geq 1}$ different from the minimal and maximal parameter sequences.  For all the  information concerning positive chain sequences that we use in this manuscript we refer to   Chihara \cite{Chihara-book}. 

Let  $\Fi_0(\xt) = P_0(\xt)$,  
\begin{equation} \label{Eq-Rationals-1}
      \Fi_n(\xt) =  \frac{(-1)^{n}}{(\xt-i)^{n} \prod_{j=1}^{n}\sqrt{d_{j+1}}} P_n(\xt) , \quad n \geq 1,
\end{equation} 
and $\mathbf{u}_{n}(\xt) = [\Fi_0(\xt), \Fi_1(\xt), \ldots, \Fi_{n-1}(\xt)]^{T}$, where $\{P_{n}\}_{n \geq 0}$ and $\{d_{n+1}\}_{n \geq 1}$ are as in \eqref{Eq-Special-R2Type-RR}.  Then the following theorem can be stated.  

\begin{theo} \label{Thm-GenEigenProblem} For any $n \geq 2$, 
\begin{equation} \label{Eq-Matrix-equality}
   \mathbf{A}_{n} \mathbf{u}_{n}(\xt) = \xt\,  \mathbf{B}_{n}  \mathbf{u}_{n}(\xt)+ \sqrt{d_{n+1}}(\xt-i) \Fi_{n}(\xt)\,\mathbf{e}_{n}, 
\end{equation}
where the $n \times n$ tridiagonal matrices $\mathbf{A}_{n}$, $\mathbf{B}_{n}$ are respectively 

\[
    \left[ 
       \begin{array}{cccccc}
        c_1  & \! i\sqrt{d_2} & 0  & \!\!\! \cdots & 0 & 0  \\[1ex]
        \!\!\! -i\sqrt{d_2}  & c_2  & \!\! i\sqrt{d_3} & \!\!\! \cdots & 0 & 0 \\[1ex]
         0  & \!\!\!\!\! -i\sqrt{d_3}  & c_3 & \!\!\! \cdots & 0 & 0  \\[1ex]
         \vdots  & \vdots  & \vdots &    & \vdots & \vdots \\[1ex]
          0  & 0  & 0 & \!\!\! \cdots & c_{n-1} & \!\! i\sqrt{d_{n}}  \\[1ex]
          0  & 0  & 0 & \!\!\! \cdots & \!\!\!\!\!  -i\sqrt{d_{n}} & c_{n} 
      \end{array} \! \right], 
\quad 
    \left[ 
       \begin{array}{cccccc}
        1 & \!\sqrt{d_2} & 0 & \cdots & 0 & 0  \\[1ex]
        \!\sqrt{d_2} & 1 &  \!\sqrt{d_3} & \cdots & 0 & 0 \\[1ex]
         0 & \!\sqrt{d_3} & 1 & \cdots & 0 & 0  \\[1ex]
         \vdots & \vdots & \vdots &    & \vdots & \vdots \\[1ex]
          0 & 0 & 0 & \cdots & 1 & \!\sqrt{d_{n}}  \\[1ex]
          0 & 0 & 0 & \cdots & \!\sqrt{d_{n}} & 1 \\[1ex]
      \end{array} \right],
\] \\[0ex]
and $\mathbf{e}_n$ is the $n^{th}$ column of the $n \times n$ identity matrix. Moreover,  the zeros of $P_{n}$ ($n \geq 2$)  are the eigenvalues of the generalized eigenvalue problem  
\begin{equation} \label{Eq-Gen-Eigenproblem}
   \mathbf{A}_{n} \mathbf{u}_{n}(\xt)= \xt\, \mathbf{B}_{n} \mathbf{u}_{n}(\xt).  
\end{equation} 
\end{theo}  

\vspace{1ex}

The generalized eigenvalue problem given by Theorem \ref{Thm-GenEigenProblem} forms the basis for our study in this manuscript.

Given a nontrivial positive measure $\mu$ on the unit circle $\T: = \{\zeta=e^{i\theta}\!\!: \, 0 \leq \theta \leq 2\pi \}$   let  $\{\Phi_{n}\}_{n \geq 0}$ be the associated sequence of monic orthogonal polynomials. The Verblunsky coefficients  $\alpha_{n} = -\overline{\Phi_{n+1}(0)}$, $n \geq 0$, which satisfy $|\alpha_{n}| < 1$, are known to play an important role in the theory of orthogonal polynomials on the unit circle.   For basic information concerning the theory of orthogonal polynomials on the unit circle we refer to \cite{Ismail-Book, Simon-Book-p1, Szego-Book}.  

A monic para orthogonal polynomial on the unit circle of degree $n$ is given by  $z\Phi_{n-1}(z) \linebreak  - \tau \Phi^{\ast}_{n-1}(z)$, where  $|\tau| =1$ and $\Phi^{\ast}_{n}(z) = z^n \overline{\Phi_{n}(1/\overline{z})}$. The name para-orthogonal polynomials first appeared in Jones, Nj\aa stad and Thron  \cite{JoNjTh-1989}, where it was also shown  that the zeros of these polynomials are simple and are of modulus equal to one. For some of the  contributions that bring  interesting properties of the zeros of these polynomials we cite \cite{CMV-EJA2002}, \cite{Golinskk-AMH2002}, \cite{Simon-JMAA2007} and \cite{Wong-JAT2007}. One of the well known  applications of para-orthogonal polynomials is with  regards to quadrature formulas on the unit circle and, in \cite{BracLiRanga-2005} and \cite{DarNjaVAssh-RMJ2003}, such quadrature formulas have been used to study the frequency analysis problem.

The sequence of polynomials $\{P_{n}\}_{n \geq 0}$ given by the  $R_{II}$ type recurrence \eqref{Eq-Special-R2Type-RR}  will be shown to be related to a certain sequence $\{z\Phi_{n-1}(z) - \tau_{n} \Phi^{\ast}_{n-1}(z)\}_{n \geq 0}$ of para-orthogonal polynomials on the unit circle. Thus, the generalized eigenvalue problem given by Theorem \ref{Thm-GenEigenProblem} provide us with another approach to study the properties of  measures and orthogonal polynomials on the unit circle.

With respect to the orthogonality \eqref{Eq-IsmailMasson-Orthogonality} associated with the $R_{II}$ type recurrence \eqref{Eq-Special-R2Type-RR} the following can be stated. 

\begin{theo} \label{Thm-Special-R2Type-MF}  Consider the  $R_{II}$ type recurrence \eqref{Eq-Special-R2Type-RR}, where we assume $\{d_{n+1}\}_{n \geq 1}$ to be  a positive chain sequence having multiple parameter sequences. Let $\{M_{n+1}\}_{n \geq 0}$ be the  maximal parameter sequence of $\{d_{n+1}\}_{n \geq 1}$ and let $\nu$ be the probability measure on the unit circle such that its Verblunsky coefficients are

\[
     \beta_{n-1} = \frac{1}{\tau_{n-1}}\,\frac{1-2M_{n} - i c_n}{1 - i c_n},  \quad n \geq 1,
\]
where $\tau_{0}= 1$ and $\tau_{n} =  \tau_{n-1}(1-ic_n)/(1+ic_n)$, $n \geq 1$.
If $d\varphi(\xt) = - d\nu((\xt+i)/(\xt-i))$ then 
\[
      \int_{-\infty}^{\infty} \xt^{j} \, \frac{P_{n}(\xt)}{(\xt^2+1)^{n}} \,  d \varphi(\xt) = \gamma_n \delta_{n,j}, \quad j=0,1, \ldots, n,
\]
where $\gamma_{0} = \int_{-\infty}^{\infty} d\varphi(\xt) = \int_{\T} d \nu(\zeta) = 1$ and $\gamma_{n} = (1- M_{n})\gamma_{n-1}$, $n \geq 1$.   
\end{theo} 

\vspace{1ex}

Apart from the results that we have presented in this manuscript, a further important advantage that we could point out is the availability of nice numerical techniques for the generation of eigenvalues    of our generalized eigenvalue problem. Subsequently, from these eigenvalues  one could easily determine the zeros of the respective para-orthogonal polynomials.

The manuscript is organized as follows. The proof of Theorem \ref{Thm-GenEigenProblem} and some properties of the zeros of the polynomials $P_n$ are given in Section \ref{Sec-InitialResults}. In Section \ref{Sec-EVP-to-Orthogonality}, using techniques from  linear algebra,  we obtain the relation between the polynomials $P_n$ and  orthogonal polynomials on the unit circle. In Section \ref{Sec-OPUC-to-POPUC2} we show how the results obtained in Section \ref{Sec-EVP-to-Orthogonality} are related to those established in \cite{BracRangaSwami-2016} and use this to obtain information concerning the orthogonality satisfied by the polynomials $P_{n}$. In particular, in this section we also give the proof of Theorem \ref{Thm-Special-R2Type-MF}. Finally, in Section \ref{Sec-Examples} examples are given to justify the results. 

\setcounter{equation}{0}
\section{Polynomials from the special $R_{II}$ recurrence } \label{Sec-InitialResults}

In this section we will consider some properties of the polynomials $P_{n}$, $n \geq 1$, given by the  $R_{II}$ recurrence relation \eqref{Eq-Special-R2Type-RR}. 

\begin{theo}  \label{Thm-LeadingCoefficient}
  The polynomial $P_n$ is of exact degree $n$ with positive leading coefficient. Precisely, if  we denote by  $\mathfrak{p}_n$ the leading coefficient of  $P_n(z)$ then $\mathfrak{p}_0 =1$, $\mathfrak{p}_1 = 1$  and 
\[ 
      0 < (1-\ell_n) = \frac{\mathfrak{p}_{n}}{\mathfrak{p}_{n-1}} < 1,  \quad  n \geq 2.
\] 
Here, $\{\ell_{n+1}\}_{n \geq 0}$ is the minimal parameter sequence of the positive chain sequence $\{ d_{n+1}\}_{n \geq 1}$. 

\end{theo}

\begin{proof}[\bf Proof]
Clearly, $\mathfrak{p}_0 =1$ and $\mathfrak{p}_1 = 1$ and one can directly verify from \eqref{Eq-Special-R2Type-RR} that  $\mathfrak{p}_2 = 1 - d_2 = 1 - \ell_2 > 0$.   Hence, the theorem holds for $n =2$. 

Assume that $\mathfrak{p}_{n-1} > 0$,  $\mathfrak{p}_{n} > 0$ and $1 - \ell_{n} = \mathfrak{p}_{n}/\mathfrak{p}_{n-1}$.   Observe that \eqref{Eq-Special-R2Type-RR} also leads to 
\[
       \frac{\xt^2+1}{(\xt-c_{n+1})(\xt-c_{n})} d_{n+1} = \frac{P_n(\xt)}{(\xt-c_n)P_{n-1}(\xt)}\Big(1 - \frac{P_{n+1}(\xt)}{(\xt-c_{n+1})P_{n}(\xt)}\Big), \quad n \geq 1.
\] 
Hence, by letting $\xt \to + \infty$ in the above equation we have 
\[
      d_{n+1} = (1-\ell_{n})\,  \lim_{\xt \to \infty} \Big(1 - \frac{P_{n+1}(\xt)}{(\xt-c_{n+1})P_{n}(\xt)}\Big).
\]
Since $\{\ell_{n+1}\}_{n \geq 0}$ is a parameter sequence of the positive chain sequence $\{ d_{n+1}\}_{n \geq 1}$, we thus have
\[
      0\  < \ \lim_{\xt \to \infty} \Big(1 - \frac{P_{n+1}(\xt)}{(\xt-c_{n+1})P_{n}(\xt)}\Big) = \ell_{n+1} \ < \ 1.
\]
Hence, $\mathfrak{p}_{n+1} \neq 0$ and $\mathfrak{p}_{n+1}/\mathfrak{p}_{n} = 1 - \ell_{n+1} > 0$. The theorem follows by induction. 
\end{proof} 

Now we give the proof of Theorem \ref{Thm-GenEigenProblem}. 

\begin{proof}[\bf Proof of Theorem \ref{Thm-GenEigenProblem}]   From \eqref{Eq-Special-R2Type-RR} we have for any $n \geq 2$, 
\[  
  \begin{array}l
    (\xt-i)\sqrt{d_2}\, \Fi_{1} + (\xt-c_{1})\Fi_{0} = 0 ,  \\[1ex]
    (\xt-i)\sqrt{d_{k+1}}\, \Fi_{k} + (\xt-c_{k})\Fi_{k-1} + (\xt+i)\sqrt{d_{k}}\, \Fi_{k-2} = 0, \quad k=2,3, \ldots, n,
   \end{array}
\]
where $\Fi_{j}$ refers to $\Fi_{j}(\xt)$. This we can write in the equivalent form 
\[
   \begin{array}{rl}
      c_{1}\, \Fi_{0} + i \sqrt{d_{2}}\, \Fi_{1} &\!\! = \xt\, [\Fi_{0} + \sqrt{d_{2}}\, \Fi_{1}], \\[1ex]
     -i \sqrt{d_{k}}\,\Fi_{k-2} + c_{k}\, \Fi_{k-1} + i \sqrt{d_{k+1}}\, \Fi_{k} &\!\! = \xt\, [ \sqrt{d_{k}}\, \Fi_{k-2} +  \Fi_{k-1} +  \sqrt{d_{k+1}}\, \Fi_{k} ], 
   \end{array}
\]
for $k = 2, 3, \ldots, n$, which is exactly the system of equations given by \eqref{Eq-Matrix-equality}. 

Observe that since $P_{n}(i) = \prod_{j=1}^{n}(i-c_{j})$, the point $i$ is not a zero of $P_{n}(\xt)$.  Moreover, the $n$ zeros of the polynomial $P_{n}(\xt)$ are the only zeros of $(\xt-i)\Fi_{n}(\xt)$. Hence,  we obtain from \eqref{Eq-Matrix-equality} that the zeros of $P_{n}$ are the eigenvalues of the generalized eigenvalue problem  \eqref{Eq-Gen-Eigenproblem}.
\end{proof}

The $n \times n$  matrices $\mathbf{A}_{n}$ and $\mathbf{B}_{n}$ in Theorem \ref{Thm-GenEigenProblem} are Hermitian and, since $\{d_{n+1}\}_{n \geq 1}$ is a positive chain sequence, the matrix  $\mathbf{B}_{n}$ is also positive definite.  We recall from  \cite[p.\,75]{Wall-Book} the positive definiteness of $\mathbf{B}_n$ can be easily verified from the identity 
\[
    \mathbf{v}_{n}^{H} \mathbf{B}_n \mathbf{v}_{n} = \sum_{j=1}^{n-1}\big|\sqrt{1-\ell_{j}}\, v_{j} - \sqrt{\ell_{j+1}} \,v_{j+1}\big|^2 \ + (1-\ell_{n})\, |v_{n}|^2,
\]
where $\mathbf{v}_{n} = [v_1, v_2, \ldots, v_{n}]^{T}$.  Another  way (see \cite{Wilkinson-book}) to confirm the positive definiteness of $\mathbf{B}_n$ is  its Cholesky decomposition $\mathbf{L}_n\mathbf{L}_n^{T}$ given in Section \ref{Sec-EVP-to-Orthogonality}. 

Since $\mathbf{B}_n$ is positive definite, from \eqref{Eq-Gen-Eigenproblem} we have 
\[  
     \xt_j^{(n)} = \frac{\mathbf{u}_{n}(\xt_j^{(n)})^{H}\mathbf{A}_{n}\, \mathbf{u}_{n}(\xt_j^{(n)}) }{ \mathbf{u}_{n}(\xt_j^{(n)})^{H}\mathbf{B}_{n} \, \mathbf{u}_{n}(\xt_j^{(n)} )},  \quad j=1,2,\ldots,n,
\]
where $x_{j}^{(n)}$ are the eigenvalues of  the generalized eigenvalue problem \eqref{Eq-Gen-Eigenproblem}. Thus, one can verify that the eigenvalues of the generalized eigen-system \eqref{Eq-Gen-Eigenproblem} (i.e., the zeros of $P_n$)  are all real (see, for example, \cite[Chap.\,5]{BDDRV-2000}).  The next theorem shows that we can say more about the zeros of the polynomials $P_n$. 

\begin{theo}  \label{Thm-InterlacingZeros}
The zeros $\xt_j^{(n)}$, $j =1,2, \ldots, n$ of $P_n$ are real and simple.  Assuming  the ordering $\xt_{j}^{(n)} < \xt_{j-1}^{(n)}$, $1 \leq j \leq n$, for the zeros, we also have the interlacing property 
\[
      \xt_{n+1}^{(n+1)} < \xt_{n}^{(n)} <  \xt_{n}^{(n+1)} < \cdots <  \xt_{2}^{(n+1)} < \xt_{1}^{(n)} <  \xt_{1}^{(n+1)},
\]
for $n \geq 1$. 
\end{theo}

\begin{proof}[\bf Proof]
Let us consider the Wronskians 
\begin{equation} \label{Eq-Wronskians}
    \mathcal{G}_n(\xt) = P_{n}^{\prime}(\xt) P_{n-1}(\xt) - P_{n-1}^{\prime}(\xt) P_{n}(\xt),  \quad n \geq 1. 
\end{equation}
From the recurrence \eqref{Eq-Special-R2Type-RR} we find 
\begin{equation} \label{Eq-WronskianRelation}
    \mathcal{G}_{n+1}(\xt) = P_n(\xt) [ P_n(\xt) - 2 d_{n+1} \xt P_{n-1}(\xt)] + d_{n+1} [\xt^2 + 1] \mathcal{G}_{n}(\xt), \quad n \geq 1. 
\end{equation}
The only zero of $P_1$ is $\xt_1^{(1)} = c_1$, which is real. Since the leading coefficient of $P_1$ is positive, we then have from \eqref{Eq-Wronskians} that $\mathcal{G}_{1}(\xt_1^{(1)}) = P_{1}^{\prime}(\xt_{1}^{(1)})P_{0}(\xt_{1}^{(1)}) > 0$.  Hence, from \eqref{Eq-WronskianRelation} 
\[
    -P_{1}^{\prime}(\xt_{1}^{(1)}) P_{2}(\xt_1^{(1)}) = \mathcal{G}_{2}(\xt_{1}^{(1)}) = d_{2} [ (\xt_1^{(1)})^2+ 1]\,\mathcal{G}_{1}(\xt_1^{(1)})  > 0. 
\]
Thus, with the leading coefficient of $P_2$ positive we conclude that the zeros $\xt_{1}^{(2)}$ and $\xt_{2}^{(2)}$ of $P_2$ can be ordered such that $\xt_{2}^{(2)} < \xt_{1}^{(1)} < \xt_{1}^{(2)}$. 

Now suppose that the zeros of $P_n$ and $P_{n-1}$ interlace as stated in the theorem. Then since the leading coefficients of $P_n$ and $P_{n-1}$ are positive we have from \eqref{Eq-Wronskians} that 
\[
     \mathcal{G}_{n}(\xt_j^{(n)}) = P_{n}^{\prime}(\xt_j^{(n)}) P_{n-1}(\xt_j^{(n)}) > 0 , \quad j = 1,2, \ldots n. 
\]
Consequently, from \eqref{Eq-WronskianRelation},
\[
      - P_{n}^{\prime}(\xt_j^{(n)}) P_{n+1}(\xt_j^{(n)}) =  \mathcal{G}_{n+1}(\xt_j^{(n)})  = d_{n+1}  [ (\xt_{j}^{(n)})^2 + 1]\, \mathcal{G}_{n}(\xt_j^{(n)}) > 0, 
\]
for $ j = 1,2, \ldots n$.  This leads to the required interlacing of the zeros of $P_n$ and $P_{n+1}$.  Hence, we conclude the proof of the theorem by induction. 
\end{proof}

Now if the matrix $\mathbf{A}_n$ is  also positive definite (negative definite) then the zeros of $P_n$ are all positive (negative).

\begin{theo}
A sufficient condition for all the zeros of $P_n$ ($n \geq 2$) to be positive  is that 
\[
        c_1 > \sqrt{d_2}, \quad c_{j} > \sqrt{d_j} + \sqrt{d_{j+1}}, \  j =2 , 3, \ldots, n-1 \quad \mbox{and} \quad c_{n} > \sqrt{d_{n}}. 
\]
Likewise, a sufficient condition for all the zeros of $P_{n}$ to be negative  is that 
\[
        c_1 < -\sqrt{d_2}, \quad c_{j} < -\sqrt{d_j} - \sqrt{d_{j+1}}, \  j =2 , 3, \ldots, n-1 \quad \mbox{and} \quad c_{n} < -\sqrt{d_{n}}. 
\]
\end{theo}

\begin{proof}[\bf Proof]

We prove  this by verifying that the Hermitian matrix $\mathbf{A}_n$ has all its eigenvalues positive (negative).  So the proof is a simple application of the Gershgorin Theorem. 
\end{proof}

\setcounter{equation}{0}
\section{Orthogonality from the generalized eigenvalue problem} \label{Sec-EVP-to-Orthogonality}

As part of the proof of Theorem \ref{Thm-InterlacingZeros}  we also find that $\mathcal{G}_{n}(\xt_{j}^{(n)}) > 0$ and $\mathcal{G}_{n+1}(\xt_{j}^{(n)}) > 0$, for $j =1 ,2 , \ldots, n$ and $n \geq 1$. The following theorem shows that $\mathcal{G}_{n}(\xt) > 0$ for all real $\xt$. 

\begin{theo}  \label{Thm-Positive-Wronskian}
With  $\xt$ and $\ys$ real, let 
\[
     G_{n}(\xt,\ys) = \frac{P_{n}(\xt) P_{n-1}(\ys) - P_{n-1}(\xt) P_{n}(\ys)}{\xt-\ys}, \quad n \geq 1.
\]
Then $G_{1}(\xt,\ys) = 1$ and 
\[ 
  \begin{array}l
    \dsp \frac{G_{n}(\xt,\ys)}{(\xt-i)^{n-1}(\ys+i)^{n-1} d_2d_3 \cdots d_{n}} = \mathbf{u}_n(\ys)^{H} \mathbf{B}_{n} \mathbf{u}_{n}(\xt),
  \end{array}
\]
for $n \geq 2$. Moreover, for the Wronskians  $\mathcal{G}_{n}(\xt)$ given by \eqref{Eq-Wronskians} there hold
\[
 \begin{array}{l}
   \dsp   \frac{\mathcal{G}_{n}(\xt)}{(\xt^2+1)^{n-1} d_2 d_3 \cdots d_{n}}   = \mathbf{u}_n(\xt)^{H} \mathbf{B}_{n}   \mathbf{u}_n(\xt) > 0, 
 \end{array} 
\]
for $n \geq 2$. 
\end{theo}

\begin{proof}[\bf Proof]
Post-multiplication of \eqref{Eq-Matrix-equality} by the conjugate transpose of $\mathbf{u}_{n}(\ys)$ gives
\begin{equation} \label{Eq-traceApp1}
  \mathbf{A}_{n} \mathbf{u}_{n}(\xt)\, \mathbf{u}_{n}(\ys)^{H} = \xt\,  \mathbf{B}_{n} \mathbf{u}_{n}(\xt)\, \mathbf{u}_{n}(\ys)^{H}+ \sqrt{d_{n+1}}(\xt-i) \Fi_{n}(\xt)\,\mathbf{e}_{n}\, \mathbf{u}_{n}(\ys)^{H}.
\end{equation}
Likewise, pre-multiplication of  the conjugate transpose of \eqref{Eq-Matrix-equality} given in terms of $\ys$ by $\mathbf{u}_{n}(\xt)$ gives
\begin{equation} \label{Eq-traceApp2}
  \mathbf{u}_{n}(\xt)\,\mathbf{u}_{n}(\ys)^{H}\mathbf{A}_{n}  = \ys\, \mathbf{u}_{n}(\xt)\,\mathbf{u}_{n}(\ys)^{H} \mathbf{B}_{n} + \sqrt{d_{n+1}}(\ys+i) \overline{\Fi_{n}(\ys)}\, \mathbf{u}_{n}(\xt)\,\mathbf{e}_{n}^{T}.
\end{equation}
It is known that given any two matrices $\mathbf{M}$ and $\mathbf{N}$ if the  products $\mathbf{M}\mathbf{N}$ and $\mathbf{N}\mathbf{M}$ exist then $Tr(\mathbf{M}\mathbf{N}) = Tr(\mathbf{N}\mathbf{M})$.  Hence, by observing the equality in the  traces of the matrices on the left hand sides of \eqref{Eq-traceApp1} and \eqref{Eq-traceApp2}, we have 
\[ 
  \begin{array}l
    \dsp -\sqrt{d_{n+1}} \left[(\xt-i) \Fi_{n}(\xt) \overline{\Fi_{n-1}(\ys)}-(\ys+i)\overline{\Fi_{n}(\ys)} \Fi_{n-1}(\xt)\right]\\[2ex]
  \qquad \qquad  \dsp   = (\xt-\ys) \Big[\sum_{j=0}^{n-1} \overline{\Fi_{j}(\ys)} \Fi_{j}(\xt)  + \sum_{j=1}^{n-1} \sqrt{d_{j+1}}\big[\, \overline{\Fi_{j-1}(\ys)} \Fi_{j}(\xt)+ \overline{\Fi_{j}(\ys)}\Fi_{j-1}(\xt) \big]\Big].
  \end{array}
\]
We can easily verify that the expression on the above right hand side is equal to \linebreak $(\xt-\ys)\mathbf{u}_n(\ys)^{H} \mathbf{B}_{n} \mathbf{u}_{n}(\xt)$ and the term on left hand side is equal to  
\[  \begin{array}{l}
   \dsp  \frac{(\xt-\ys) G_{n}(\xt,\ys)}{(\xt-i)^{n-1}(\ys+i)^{n-1} d_2d_3 \cdots d_{n}}.
    \end{array}
\]
This proves the first part of the theorem. 

The latter part of the theorem follows from $\lim_{\ys \to \xt} G_{n}(\xt,\ys) = \mathcal{G}_{n}(\xt)$ and the positive definiteness of the matrix  $\mathbf{B}_{n}$.
\end{proof}

It is not difficult to verify that  the Cholesky decomposition $\mathbf{L}_{n} \mathbf{L}_{n}^{T}$ of the real positive definite matrix $\mathbf{B}_{n}$ is such that 
\[
   \mathbf{L}_{n} = \left[ 
       \begin{array}{cccccc}
       \sqrt{1-\ell_{1}} & 0 &  \cdots & 0 & 0  \\[1ex]
        \sqrt{\ell_2} & \!\!\!\sqrt{1-\ell_{2}} &  \cdots & 0 & 0 \\[1ex]
         \vdots & \vdots &    & \vdots & \vdots \\[1ex]
          0 & 0 &  \cdots & \!\!\!\sqrt{1-\ell_{n-1}} & 0  \\[1ex]
          0 & 0 &  \cdots & \sqrt{\ell_{n}} & \!\!\!\sqrt{1-\ell_{n}} \\[1ex]
      \end{array} \right],
\]
where $\{\ell_{n+1}\}_{n \geq 0}$ is the minimal parameter sequence of  $\{ d_{n+1}\}_{n \geq 1}$. 

Now consider the ordinary  eigenvalue problem 
\[
    \mathbf{L}_{n}^{-1} \mathbf{A}_{n} (\mathbf{L}_{n}^{-1})^{T} \mathbf{L}_{n}^{T} \mathbf{u}_{n}(\xt) = \xt \mathbf{L}_{n}^{T} \mathbf{u}_{n}(\xt), 
\]
which is equivalent to \eqref{Eq-Gen-Eigenproblem}. The eigenvectors are $\mathbf{L}_{n}^{T} \mathbf{u}_{n}(\xt_k^{(n)})$, $k=1,2, \ldots, n$ and the matrix involved  $\mathbf{L}_{n}^{-1} \mathbf{A}_{n} (\mathbf{L}_{n}^{-1})^{T}$ is Hermitian. Hence,  we must have that the eigenvectors  $\mathbf{L}_{n}^{T} \mathbf{u}_{n}(\xt_j^{(n)})$ and $\mathbf{L}_{n}^{T} \mathbf{u}_{n}(\xt_k^{(n)})$ are orthogonal to each other when $j \neq k$. We can also verify this from Theorem \ref{Thm-Positive-Wronskian} by taking $\xt = \xt_j^{(n)}$ and $\ys = \xt_k^{(n)}$.  That is, from Theorem \ref{Thm-Positive-Wronskian}, 
\[
    [\mathbf{L}_{n}^{T}\mathbf{u}_{n}(\xt_{k}^{(n)})]^{H}\, [\mathbf{L}_{n}^{T}\mathbf{u}_{n}(\xt_{j}^{(n)})] = \mathbf{u}_n(\xt_{k}^{(n)})^{H} \mathbf{B}_{n} \mathbf{u}_{n}(\xt_{j}^{(n)}) =  \frac{\mathcal{G}_{n}(\xt_{k}^{(n)})}{[(\xt_{k}^{(n)})^2+1]^{n-1} d_2 d_3 \cdots d_{n}} \, \delta_{j,k}.
\]

Thus, we consider the matrix  
\[
     \mathbf{Y}_{n} = \big[\mathbf{y}_{1}^{(n)}, \mathbf{y}_{2}^{(n)}, \ldots , \mathbf{y}_{n}^{(n)}\big] ,
\]
where 
\(
   \mathbf{y}_r^{(n)} = \sqrt{\lambda_{r}^{(n)}} \, \mathbf{L}_{n}^{T} \mathbf{u}_{n}(\xt_r^{(n)}) 
\)
and 
\begin{equation} \label{Eq-Lambda-rn}
     \lambda_{r}^{(n)} = \frac{[(\xt_r^{(n)})^2+1]^{n-1} d_2 d_3 \cdots d_{n}}{\mathcal{G}_{n}(\xt_r^{(n)})}. 
\end{equation}
Clearly $\mathbf{Y}_{n}^{H} \mathbf{Y}_{n} = \mathbf{I}_{n}$ and thus, $\mathbf{Y}_{n}^{H} = \mathbf{Y}_{n}^{-1}$.   

Now if we define the new rational functions $\hat{\Fi}_{k}(\xt)$ by 
\[
  \begin{array}{ll}
     \hat{\Fi}_{k}(\xt) &= \sqrt{\ell_{k+1}}\,\Fi_{k}(\xt) + \sqrt{1-\ell_{k}}\,\Fi_{k-1}(\xt),  \\[2ex]
     & = \frac{(-1)^{k} \sqrt{\ell_{k+1}}}{(x-i)^k \prod_{j=1}^{k}\sqrt{d_{j+1}}}\big[P_{k}(x) - (1-\ell_{k})(x-i) P_{k-1}(x)\big],
  \end{array} \quad k \geq 1,   
\]
then from $ \mathbf{Y}_{n} \mathbf{Y}_{n}^{H} = \mathbf{I}_{n}$ we have 
\begin{equation} \label{Eq-Discrete-Orthogonality-1}
   \begin{array}l
     \dsp \sum_{r=1}^{n} \lambda_{r}^{(n)}\, \overline{ \hat{\Fi}_{m}(\xt_{r}^{(n)})}\, \hat{\Fi}_{k}(\xt_{r}^{(n)})  = \delta_{m,k}, 
   \end{array}
\end{equation}
for $m,k = 1,2, \ldots, n$. 

The transformation 
\begin{equation} \label{Eq-Transformation}
     \zeta = \zeta(\xt) = \frac{\xt+i}{\xt-i},
\end{equation}
maps the real line $(-\infty, \infty)$ on to the (cut) unit circle $\T = \{ \zeta = e^{i\theta}, 0 < \theta < 2 \pi\}$.  The inverse of this transformation is $\xt(\zeta) = i(\zeta+1)/(\zeta-1)$.  

With this transformation let $\{R_{n}\}_{n \geq 0}$ be given by  $R_{0}(\zeta) = \Fi_{0}(\xt) = 1$ and 
\begin{equation} \label{Eq-P-to-R}
      R_{k}(\zeta) =  \frac{2^{k} P_{k}(\xt)}{(\xt-i)^{k}} = \Fi_{k}(\xt)\, (-2)^{k}\prod_{j=1}^{k}\sqrt{d_{j+1}}  , \quad k \geq 1.  
\end{equation}
Clearly, $R_{k}(\zeta)$ is a polynomial in $\zeta$ of exact degree $k$. To be precise, from the recurrence relation \eqref{Eq-Special-R2Type-RR} we find $R_{1}(\zeta) = (1+ic_{1})\zeta + (1-i c_{1})$ and 
\begin{equation} \label{Eq-TTRR-R}
    R_{k+1}(\zeta) = [(1+ic_{k+1})\zeta + (1-ic_{k+1})] R_{k}(\zeta) - 4 d_{k+1} \zeta R_{k-1}(\zeta), \quad k \geq 1.
\end{equation}
Hence, the leading coefficient of $R_{k}$ is $\prod_{j=1}^{k} (1+ic_{j})$. We also have
\[
     \hat{\Fi}_{k}(\xt) = \frac{(-1)^{k} \sqrt{\ell_{k+1}}}{2^{k} \prod_{j=1}^{k}\sqrt{d_{j+1}}}\, \hat{R}_{k}(\zeta), \quad k \geq 1, 
\]
where $\hat{R}_{k}(z) = R_{k}(z)- 2(1-\ell_{k}) R_{k-1}(z)$. Hence, from \eqref{Eq-Discrete-Orthogonality-1}
\begin{equation} \label{Eq-Discrete-Orthogonality-2}
  \begin{array}{l}
   \dsp  \sum_{r=1}^{n} \lambda_{r}^{(n)}\, \overline{\hat{R}_{m}(\zeta_{r}^{(n)})}\, \hat{R}_{k}(\zeta_{r}^{(n)})
   = \frac{2^{2k} \prod_{j=1}^{k}d_{j+1}}{\ell_{k+1}}\, \delta_{m,k},
  \end{array}
\end{equation}
for $m,k = 1,2, \ldots, n$, where
\[
     \zeta_{r}^{(n)} = \frac{\xt_{r}^{(n)}+i}{\xt_{r}^{(n)}-i}, \quad r =1,2, \ldots, n.
\]
Moreover, since $\lim_{\xt \to \infty} \hat{\Fi}_{k}(\xt) = 0$, we also have $\hat{R}_{k}(1)= 0$.  Hence, we introduce the polynomials 
\begin{equation} \label{Eq-OPUC-Relation1}
    \Fii_{k-1}(z)  = \frac{1}{\prod_{j=1}^{k} (1+ic_{j})}\frac{\hat{R}_{k}(z)}{z-1}, \quad k \geq 1.
\end{equation}
Note that $\Fii_{k}$ is a monic polynomial of degree $k$.  With the transformation \eqref{Eq-Transformation} one can also write  
\begin{equation} \label{Eq-OPUC-Relation2}
    \Fii_{k-1}(\zeta) = \frac{-i\,2^{k-1}}{\prod_{j=1}^{k}(1+ic_{j})}\frac{1}{(x-i)^{k-1}}\big[P_{k}(x) - (1-\ell_{k}) (x-i) P_{k-1} \big], \quad k \geq 1.
\end{equation}

\begin{theo} \label{Thm-Orthogonality1-UC}
Given the  $R_{II}$ type recurrence \eqref{Eq-Special-R2Type-RR} let $\mu$ be the positive measure on the unit circle such that its Verblunsky coefficients are 
\begin{equation} \label{Eq-Verblunsky-Characterization}
    \alpha_{n-1} = - \frac{1}{\tau_{n}}\,\frac{1-2\ell_{n+1} - i c_{n+1}}{1 - i c_{n+1}},  \quad n \geq 1,
\end{equation} 
where $\tau_{0}= 1$ and $\tau_{n} =  \tau_{n-1}(1-ic_n)/(1+ic_n)$, $n \geq 1$. The polynomials $\{\Phi_n\}_{n \geq 0}$ given by \eqref{Eq-OPUC-Relation2} are the monic orthogonal polynomials on the unit circle that satisfy 
\[  
   \int_{\T} \overline{\Fii_{m}(\zeta)}\,\Fii_{k}(\zeta)\,  d \Muii(\zeta) = \frac{(1+c_{1}^2)\,2^{2k} \prod_{j=1}^{k+1}d_{j+1}}{\ell_{k+2} \prod_{j=1}^{k+1} (1+c_{j}^2)}\, \delta_{m,k} . 
\]
Moreover, if $\{d_{n+1}\}_{n \geq 1}$ is a positive chain sequence with multiple parameter sequences then the measure $\Muii$ is such that $\int_{\T} |\zeta-1|^{-2} d \Muii(\zeta)$ exists. 

\end{theo}
\begin{remark}
 The measure $\Muii$ characterized by the Verblunsky coefficients in \eqref{Eq-Verblunsky-Characterization} may or may not be such that the integral $\int_{\T} |\zeta-1|^{-2} d \Muii(\zeta)$ exists.  However, this is definitely the case if $\{d_{n+1}\}_{n \geq 1}$ is a positive chain sequence with multiple parameter sequences.  
\end{remark}

\begin{proof}[\bf Proof of Theorem \ref{Thm-Orthogonality1-UC}] 

Observe that $z-1 = \hat{R}_{1}(z)/(1+ic_1)$.  Hence, from \eqref{Eq-Discrete-Orthogonality-2}, 
\[
      \sum_{r=1}^{n} \lambda_{r}^{(n)} |\zeta_{r}^{(n)} - 1|^2 = \frac{1}{1+c_1^2} \sum_{r=1}^{n} \lambda_{r}^{(n)} |\hat{R}_{1}(\zeta_{r}^{(n)})|^2 = \frac{4}{1+c_1^2}.
\]
Thus, we rewrite \eqref{Eq-Discrete-Orthogonality-2} as   
\[ 
  \begin{array}{l}
   \dsp  \sum_{r=1}^{n} \hat{\lambda}_{r}^{(n)}\, \overline{\Fii_{m}(\zeta_{r}^{(n)})}\,\Fii_{k}(\zeta_{r}^{(n)})
   = \frac{(1+c_1^2)\,2^{2k} \prod_{j=1}^{k+1}d_{j+1}}{\ell_{k+2} \prod_{j=1}^{k+1} (1+c_{j}^2)}\, \delta_{m,k}, 
  \end{array}
\]
for $m, k = 0,1, \ldots,n-1$, where 
\[
     \hat{\lambda}_{r}^{(n)} = \frac{1}{4} (1+c_1^2)\,|\zeta_{r}^{(n)} - 1|^2\lambda_{r}^{(n)} > 0, \quad r=1,2,\ldots, n
\]
and $\sum_{r=1}^{n} \hat{\lambda}_{r}^{(n)} = 1$. Therefore,  if we introduce the series of probability measures 
\[
      \Lambda_{n}(\zeta) = \sum_{r=1}^{n} \hat{\lambda}_{r}^{(n)}\, \delta_{\zeta_{r}^{(n)}}
\]
where $\delta_{w}$ is the Dirac measure concentrated at the point $w$, then 
\[
   \int_{\T} \overline{\Fii_{m}(\zeta)}\,\Fii_{k}(\zeta)\,  d \Lambda_{n}(\zeta) =  \frac{(1+c_1^2)\, 2^{2k} \prod_{j=1}^{k+1}d_{j+1}}{\ell_{k+2} \prod_{j=1}^{k+1} (1+c_{j}^2)}\, \delta_{m,k}, 
\]
for $m, k = 0,1, \ldots,n-1$ and $n \geq 2$. It is also not difficult to verify that by expressing the monomial $\zeta^{k}$ as a linear combination of the set of monic polynomials $\{\Fii_{j}\}_{j=0}^{k}$ that 
\[
      \int_{\T}{\zeta^k} d \Lambda_{n}(\zeta) = \int_{\T}{\zeta^k} d \Lambda_{n+m}(\zeta) \quad \mbox{and} \quad \int_{\T}{\overline{\zeta}^k} d \Lambda_{n}(\zeta) = \int_{\T}{\overline{\zeta}^k} d \Lambda_{n+m}(\zeta), 
\]
for $n > k$ and $m \geq 1$.  Thus, we could conclude that  the sequence $\{\Lambda_{n}\}$ has a {\em subsequence} that converges to a probability measure $\mu$ (see, for example, \cite{Billingsley-Book}) and that 
\[
      \int_{\T}{\zeta^k} d \Lambda_{n}(\zeta) = \int_{\T}{\zeta^k} d \Muii(\zeta) \quad \mbox{and} \quad \int_{\T}{\overline{\zeta}^k} d \Lambda_{n}(\zeta) = \int_{\T}{\overline{\zeta}^k} d \Muii(\zeta), 
\]
for $n > k$.  Thus,  the orthogonality relation given in  the theorem holds. The Verblunsky coefficients are then obtain from $-\overline{\Phi_n(0)}$.  

Now to prove the latter part of the theorem, note that
\[ 
  \begin{array}{l}
   \dsp \int_{\T}\frac{1}{|\zeta-1|^2} d \Lambda_{n}(\zeta)  = \sum_{r=1}^{n} \hat{\lambda}_{r}^{(n)} \frac{1}{|\zeta_{r}^{(n)}-1|^2}   = \frac{1}{4} (1+c_1^2)  \sum_{r=1}^{n} \lambda_{r}^{(n)} , \quad n \geq 2.
 \end{array}
\]
Since one can verify that the set of polynomials  $\{\hat{R}_{1}, \hat{R}_{2}, \ldots, \hat{R}_{n-1}, R_{n-1}\}$ form a basis in $\Pset_{n-1}$,  by considering the linear combination
\[
     1 = b_1 \hat{R}_{1}(z) + b_2 \hat{R}_{2}(z) + \ldots + b_{n-1} \hat{R}_{n-1}(z) - 2b_{n} (1-\ell_n) R_{n-1}(z),
\]
we find 
\begin{equation}  \label{Eq-Cobination-for-1}
   b_k= \frac{-1}{2^{k}(1-\ell_1)\cdots (1-\ell_{k})}, \quad k=1,2, \ldots, n.
\end{equation}
Hence, by observing   
\[
1 = \sum_{k=1}^{n-1} b_{k} \hat{R}_{k}(\zeta_{r}^{(n)}) \ - 2b_{n}(1-\ell_{n})R_{n-1}(\zeta_{r}^{(n)}) = \sum_{k=1}^{n} b_{k} \hat{R}_{k}(\zeta_{r}^{(n)}) = \sum_{k=1}^{n} b_{k} \overline{\hat{R}_{k}(\zeta_{r}^{(n)})}, 
\]
we can write  
\[ 
  \begin{array}{l}
   \dsp \int_{\T}\frac{1}{|\zeta-1|^2} d \Lambda_{n}(\zeta)   = \frac{1}{4} (1+c_1^2)  \sum_{r=1}^{n} \lambda_{r}^{(n)} \Big[ \sum_{k=1}^{n} b_{k} \overline{\hat{R}_{k}(\zeta_{r}^{(n)})}\sum_{k=1}^{n} b_{k} \hat{R}_{k}(\zeta_{r}^{(n)}) \Big], \quad n \geq 2. 
 \end{array}
\]
Thus, by using the orthogonality \eqref{Eq-Discrete-Orthogonality-2} and then the values for $b_{k}$  in \eqref{Eq-Cobination-for-1},
\[
    \begin{array}{ll}
        \dsp \int_{\T}\frac{1}{|\zeta-1|^2} d \Lambda_{n}(\zeta)  & \dsp = \frac{1}{4} (1+c_1^2) \sum_{k=1}^{n}b_{k}^2 \sum_{r=1}^{n} \lambda_{r}^{(n)}|\hat{R}_{k}(\zeta_{r}^{(n)})|^2  \\[2.5ex]
       &\dsp  = \frac{1}{4} (1+c_1^2) \Big[1 + \sum_{k=2}^{n} \prod_{j=2}^{k}\frac{\ell_{j}}{(1-\ell_{j})} \Big], \quad n \geq 2.
     \end{array}
\]
By Wall's criteria  for the maximal parameter sequence of a positive chain sequence (see \cite[p.101]{Chihara-book}), the series $\sum_{k=2}^{\infty} \prod_{j=2}^{k}\frac{\ell_{j}}{(1-\ell_{j})}$ converges if $\{ d_{n+1}\}_{n \geq 1}$ is a multiple parameter positive chain sequence and diverges if $\{ d_{n+1}\}_{n \geq 1}$ is a single parameter positive chain sequence. Thus, if $\{ d_{n+1}\}_{n \geq 1}$ is a multiple parameter positive chain sequence then 
\[
        \dsp \int_{\T}\frac{1}{|\zeta-1|^2} d \mu(\zeta)    = \frac{1}{4} (1+c_1^2) \Big[1 + \sum_{k=2}^{\infty} \prod_{j=2}^{k}\frac{\ell_{j}}{(1-\ell_{j})} \Big] 
\]
is finite, which concludes the second part of the theorem. 
\end{proof}

\begin{remark}
  The initial part of Theorem \ref{Thm-Orthogonality1-UC} is a kind of Favard type theorem on the unit circle, where a measure $\mu$ is derived from the pair of sequences $[\{c_n\}_{n \geq 1}, \{d_{n+1}\}_{n \geq 1}]$.  See an alternative proof of this in \cite{BracRangaSwami-2016}.  The classical Favard theorem on the unit circle is where one starts with the Verblunsky coefficients to derive the measure $\mu$. See \cite{ENZG1} for a simple proof the classical Favard theorem on the unit circle.  
\end{remark}

We now state the above theorem in terms of the rational functions $\hat{\Fi}_{n}(\xt) = \linebreak \sqrt{\ell_{n+1}}\, \Fi_{n}(\xt) + \sqrt{1-\ell_{n}}\, \Fi_{n-1}(\xt)$, $ n \geq 1$. 

\begin{coro} Let the positive measure $\Muii$ be as in Theorem \ref{Thm-Orthogonality1-UC} and let $\psi$, defined on the real line, be given by $d\psi(\xt) = -d\Muii((\xt+i)/(\xt-i))$. Then the following hold.   
\[
     \int_{-\infty}^{\infty} \hat{\Fi}_{m}(\xt)\, \hat{\Fi}_{k}(\xt)\, (\xt^2+1) \, d \psi(\xt) = (1+c_1^2) \delta_{m,k}, 
\]
for $m, n = 1, 2, \ldots$.

\end{coro}

\begin{remark}
The sequence of polynomials $\{R_n\}_{n \geq 1}$ are in fact a sequence of para-orthogonal polynomials  associated with the sequence of orthogonal polynomials on the unit circle $\{\Phi_{n}\}_{n \geq 0}$. 
\end{remark}

\begin{remark}
According to the results shown in \cite{Ismail-Masson-JAT1995}, there is a moment functional $\mathcal{L}$ such that  $\mathcal{L}[\xt^{k}P_n(\xt)] = 0$ for $0 \leq k < n$.  What can we say about moment functional $\mathcal{L}$? 
\end{remark}

The above remarks  will become clearer with results presented in the next section.    

\setcounter{equation}{0}
\section{Para-orthogonal polynomials to $R_{II}$ type recurrence  }  \label{Sec-OPUC-to-POPUC2}

From \eqref{Eq-Verblunsky-Characterization} it is easily verified that
\begin{equation} \label{Eq-Tau-Recurrence-0}
    \tau_1 =  \frac{1-ic_1}{1+ic_1} \quad \mbox{and} \quad \tau_{n+1} = \frac{\tau_n + \overline{\alpha}_{n-1}}{1+ \tau_{n}\alpha_{n-1}}, \ \ n \geq 1.
\end{equation}
Moreover, we can also verify the following  formulas for $\{c_{n}\}_{n \geq 1}$ and $\{\ell_{n+1}\}_{n \geq 0}$ in terms of the sequences $\{\tau_{n}\}_{n \geq 1}$ and $\{\alpha_{n}\}_{n \geq 0}$.
\begin{equation*}\label{***}
  \begin{array}{ll}
    \dsp  c_{1} = i\frac{\tau_1 - 1}{\tau_1 + 1}, &  \dsp c_{n+1} =  \frac{ \mathcal{I}m(\tau_{n}\alpha_{n-1})} {1 + \mathcal{R}e(\tau_{n}\alpha_{n-1})}  = \frac{ \mathcal{I}m(\tau_{n+1}\alpha_{n-1})} {1 - \mathcal{R}e(\tau_{n+1}\alpha_{n-1})},  \quad n \geq 1  \\[3ex]
    & \dsp  \ell_{n+1} = \frac{1}{2}\, \frac{|1 + \tau_{n}  \alpha_{n-1}|^2}{1 + \mathcal{R}e( \tau_{n} \alpha_{n-1})} = \frac{1}{2}\, \frac{1 - |\tau_{n+1}  \alpha_{n-1}|^2}{1 - \mathcal{R}e( \tau_{n+1} \alpha_{n-1})}, \quad n \geq 1.
  \end{array}
\]
To verify the two different expressions for $c_{n+1}$ and $\ell_{n+1}$ we require
\[
    (1- \tau_{n+1}\alpha_{n-1})(1 + \tau_{n}\alpha_{n-1}) = 1 - |\alpha_{n-1}|^2, \quad n \geq 1,
\] 
which is a consequence of the recurrence formula in \eqref{Eq-Tau-Recurrence-0}. 

Hence, given the pair of sequences $\{c_{n}\}_{n \geq 1}$ and $\{d_{n+1}\}_{n \geq 1}$, the probability measure $\Muii$ on the unit circle given by Theorem \ref{Thm-Orthogonality1-UC} is unique.  That is, the measure $\Muii$ for which the polynomials $\Fii_{n}$ given by \eqref{Eq-OPUC-Relation1} are orthogonal, is unique. 

However, the reciprocal of the above observation is not true. There are infinitely many pairs of sequences $\{c_{n}\}_{n \geq 1}$ and $\{d_{n+1}\}_{n \geq 1}$ produce  the same measure $\Muii$ and the same orthogonal polynomials $\Fii_{n}$.  This follows from the following results established in \cite{BracRangaSwami-2016}.  


With any $\tau$ chosen such that $|\tau| = 1$,  $\tau \neq -1$, let the sequence of numbers $\{\tau_n\}_{n=0}^{\infty}$  be given by
\begin{equation*} \label{Eq-TauRecurence1}
   \tau_{1} = \tau \quad \mbox{and} \quad  \tau_{n+1} = \frac{\tau_{n} + \overline{\alpha}_{n-1}}{1 + \tau_{n} \alpha_{n-1}}, \quad n \geq 1,
\end{equation*}
where  $\alpha_{n-1} = - \overline{\Phi_n(0)}$, $n \geq 1$, are the Verblunsky coefficients.  Then the sequence  $\{R_n\}$ of para-orthogonal polynomials on the unit circle   given by
\[
   R_n(z) \frac{1+\tau_1}{2} \prod_{k=1}^{n-1} \frac{ 1 + \Re(\tau_{k}\alpha_{k-1})}{1 + \tau_{k} \alpha_{k-1}}  = z\Phi_{n-1}(z) + \tau_{n}\Phi_{n-1}^{\ast}(z), \quad n \geq 1,
\]
satisfy the three term recurrence formula \eqref{Eq-TTRR-R}, where  the real sequence $\{c_n\}$ and the positive chain sequence $\{d_{n+1}\}_{n=1}^{\infty} $ are such that
\begin{equation*}\label{Eq-CoeffsTTRR-11}
     c_{1}  = i\frac{\tau_1 - 1}{\tau_1 + 1}, \quad c_{n+1} =  \frac{ \Im(\tau_{n}\alpha_{n-1})} {1 + \Re(\tau_{n}\alpha_{n-1})}   \quad \mbox{and}  \quad d_{n+1} = (1-\ell_{n})\ell_{n+1}, \quad n \geq 1,
\end{equation*}
where  $\{\ell_{n+1}\}_{n\geq 0}$ is the minimal parameter sequence  of $\{d_{n+1}\}_{n \geq 1}$ given by
\[
    \ell_1 = 0 \quad \mbox{and} \quad   \ell_{n+1} = \frac{1}{2} \frac{|1 + \tau_{n}  \alpha_{n-1}|^2}{1 + \Re( \tau_{n} \alpha_{n-1})}, \quad n \geq 1.
\]
Moreover, the  measure $\Muii$ is such that the integral $\int_{\mathbb{T}} |\zeta-1|^{-2}d\Muii(\zeta)$  exists if and only if  there exists a $\tau$ $(|\tau| = 1, \tau \neq -1)$ such that the corresponding positive chain sequence $\{d_{n+1}\}_{n \geq 1}$ is not a single parameter positive chain sequence.

\begin{remark}
  The above results  are  precisely those given in \cite[Thm.\,4.1]{BracRangaSwami-2016}, where the sequence $\{\rho_{n}\}_{n \geq 0}$ that appear in \cite[Thm.\,4.1]{BracRangaSwami-2016} is such that $\rho_{n} = - \tau_{n+1}$, $n \geq 0$.    
\end{remark} 

The following Theorem, also given  in \cite{BracRangaSwami-2016}, is a restatement of the above results under the additional assumption  that  the measure $\Muii$ belongs to the class of measures for which the principal value integral 

\begin{equation} \label{Eq-PVI-UC}
       I(\Muii) = \dashint_\mathcal{\T} \frac{\zeta}{\zeta-1} d \Muii(\zeta) = \lim_{\epsilon \to 0} \int_{\epsilon}^{2\pi - \epsilon} \frac{e^{i\theta}}{e^{i\theta}-1} d\Muii(e^{i \theta}),
\end{equation}
exists. Observe that this class of measures also includes measures for which the integral  $\int_{\mathbb{T}} |\zeta-1|^{-2}d\Muii(\zeta)$ does not exist.  The simplest example of such measures is the Lebesque measure $d\mu(\zeta) = (2\pi i \zeta)^{-1} d\zeta$, where $I = 1/2$ but $\int_{\mathbb{T}} |\zeta-1|^{-2}(2\pi i \zeta)^{-1} d\zeta$ does not exist.

\begin{theo}[\cite{BracRangaSwami-2016}] \label{Thm-BRS2016-2}
  Let $\Muii$ be a nontrivial probability measure on the unit circle such that $I = I(\Muii)$ exists and  let  $\{\Phi_n\}_{n=0}^{\infty}$ be the  sequence of monic orthogonal polynomials on the unit circle with respect to $\Muii$.  With any  $s$ be such that $-\infty < s < \infty$,  let the sequence  $\{\tau_{n}(s)\}_{n=1}^{\infty}$ be given by
\[
	     \dsp \tau_{1}(s)=  \frac{I+is}{\overline{I} - is}\quad \mbox{and} \quad  \tau_{n+1}(s) = \frac{\tau_{n}(s) + \overline{\alpha}_{n-1}}{1 + \tau_{n}(s) \alpha_{n-1}}, \quad n \geq 1.
\]
Then for the sequence  $\{R_n(s;z)\}$ of para orthogonal polynomials, in $z$,  given by 
\[
   \quad R_n(s;z)  \frac{1+\tau_1(s)}{2} \prod_{k=1}^{n-1} \frac{ 1 + \Re[\tau_{k}(s)\alpha_{k-1}]}{1 + \tau_{k}(s) \alpha_{k-1}} = z\Phi_{n-1}(z) + \tau_{n}(s)\Phi_{n-1}^{\ast}(z), \quad n \geq 1,
\]
the following three term recurrence formula holds.
\begin{equation*}
   R_{n+1}(s;z) = \big[(1+ic_{n+1}(s))z + (1-ic_{n+1}(s))\big] R_{n}(s;z) - 4\,d_{n+1}(s) z R_{n-1}(s;z),
\end{equation*}
with $R_{0}(s;z) = 1$ and $R_{1}(s;z) = (1+ic_{1}(s))z + (1-ic_{1}(s))$,  where the real sequences $\{c_n(s)\}$ and $\{d_{n+1}(s)\}$ are such that
\begin{equation*}\label{Eq-CoeffsTTRR-12}
  \begin{array}l
     \dsp c_{1}(s) = i\frac{\tau_1(s) - 1}{\tau_1(s) + 1}   =  -2[s + \Im(I)]  \quad \mbox{and} \\ [2ex]
     \dsp c_{n+1}(s)=  \frac{ \Im[\tau_{n}(s)\alpha_{n-1}]} {1 + \Re[\tau_{n}(s)\alpha_{n-1}]}, \ \  d_{n+1}(s) =  [1-\ell_{n}(s)]\ell_{n+1}(s), \ \ n \geq 1.
   \end{array}
\end{equation*}
Here,  $\{d_{n+1} (s)\}_{n=1}^{\infty}$ is a positive chain sequence and   $\{\ell_{n+1}(s)\}_{n=0}^{\infty}$ is its minimal parameter sequence,  which is given by
\[
    \ell_{1}(s) = 0 \quad \mbox{and} \quad   \ell_{n+1}(s) = \frac{1}{2}\, \frac{|1 + \tau_{n}(s)  \alpha_{n-1}|^2}{1 + \Re\big( \tau_{n}(s) \alpha_{n-1}\big)}, \quad n \geq 1.
\]
With respect to the measure $\Muii$, the polynomials $R_n(s;.)$ also satisfy the orthogonality property  $\mathcal{N}^{(s)}[z^{-n+k}R_n(s;z)] = 0$, $k=0,1,\ldots, n-1$,
with respect to the moment fuctional $\mathcal{N}^{(s)}$ given by
\begin{equation} \label{Eq-TheoMain4}
     \mathcal{N}^{(s)}[f(z)]  = \dashint_{\mathbb{T}} f(\zeta) \frac{\zeta}{\zeta-1} d\Muii (\zeta) + i\,s\, f(1).
\end{equation} 

\end{theo}

\vspace{1.5ex}

The orthogonal polynomials  $\{\Phi_n(z)\}_{n=0}^{\infty}$ can be recovered from $\{R_n(s;z)\}_{n=0}^{\infty}$ by
\[
     (z-1)\Phi_n(z)\prod_{k=1}^{n+1}\big(1+ic_{k}(s)\big) = R_{n+1}(s;z) - 2\big(1-\ell_{n+1}(s)\big)R_{n}(s;z), \quad n \geq 0,
\]
which we can identify with the relation \eqref{Eq-OPUC-Relation1}.  

Observe that $I$ in Theorem \ref{Thm-BRS2016-2} is such that $I + \overline{I} = \Muii_0 = 1$.  Hence, $\mathcal{R}e(I) = 1/2$ and $\tau_{1}(s)  \neq -1$ if $-\infty < s < \infty$.

It was also shown in \cite{BracRangaSwami-2016} that the sequence $\{d_{n+1} (s)\}_{n=1}^{\infty}$ given in Theorem \ref{Thm-BRS2016-2} is always a single parameter positive chain sequence  if $s \neq 0$.  However,  $\{d_{n+1} (0)\}_{n=1}^{\infty}$ is  a single parameter positive chain sequence if and only if the integral $\int_{\mathbb{T}} |\zeta-1|^{-2}d\Muii(\zeta)$ does not exists. \\[-1ex]

\noindent {\bf Somes remarks with respect to the measure obtained in Theorem \ref{Thm-Orthogonality1-UC}}. 
\begin{itemize}

\item There are infinitely many pairs $[\{c_{n}\}_{n \geq 1}, \{d_{n+1}\}_{n \geq 1}]$, where the $\{c_{n}\}_{n \geq 1}$ is a real sequence and $\{d_{n+1}\}_{n \geq 1}$ is a positive chain sequence, considered as in  Theorem \ref{Thm-Orthogonality1-UC} produce the same measure $\Muii$.  However, as we have shown above, there is a one to one correspondence between any pair $[\{c_{n}\}_{n \geq 1}, \{d_{n+1}\}_{n \geq 1}]$ and the pair $[\{\alpha_{n}\}_{n \geq 0}, \tau]$, where $\{\alpha_{n}\}_{n \geq 0}$ is the sequence of Verblunsky coefficients associated with the measure $\mu$ and $\tau = (1-ic_1)/(1+ic_1)$. 

\item If the resulting measure $\Muii$ is such that the principal value integral $I(\Muii)$ defined by \eqref{Eq-PVI-UC} exists, then each pair  $[\{c_{n}\}_{n \geq 1}, \{d_{n+1}\}_{n \geq 1}]$ that produces $\Muii$ can be identified to be equal to a pair  $[\{c_{n}(s)\}_{n \geq 1}, \{d_{n+1}(s)\}_{n \geq 1}]$ given by Theorem \ref{Thm-BRS2016-2}, where $s = -\Im(I(\Muii)) - c_1/2$. 

\item In a pair $[\{c_{n}\}_{n \geq 1}, \{d_{n+1}\}_{n \geq 1}]$ that equals to $[\{c_{n}(s)\}_{n \geq 1}, \{d_{n+1}(s)\}_{n \geq 1}]$, where  $s \neq 0$, the positive chain sequence $\{d_{n+1}\}_{n \geq 1}$ can not be a multiple parameter positive chain sequence. However, in the pair equals to $[\{c_{n}(0)\}_{n \geq 1}, \{d_{n+1}(0)\}_{n \geq 1}]$ the sequence $\{d_{n+1}\}_{n \geq 1}$ is a multiple  parameter positive chain sequence only if $\Muii$ is such that integral $\int_{\T}|\zeta-1|^{-1} d \Muii(\zeta)$ exists.

\item Observe that if the  integral $\int_{\T}|\zeta-1|^{-1} d \Muii(\zeta)$ exists then the principal value integral $I(\Muii)$ also exists as a normal integral.  

\end{itemize}

Hence, if the measure $\Muii$ obtained from Theorem \ref{Thm-Orthogonality1-UC} is such that the principal value integral $I = I(\Muii)$ exists then from Theorem \ref{Thm-BRS2016-2} we have that the polynomials $R_{n}$ given by \eqref{Eq-TTRR-R} satisfy 
\begin{equation} \label{Eq-Lorthogonality-2} 
    \dashint_{\mathbb{T}} \zeta^{-n+k}R_{n}(\zeta)  \frac{\zeta}{\zeta-1} d\Muii (\zeta) - i\,[\Im(I)+c_{1}/2]\, R_{n}(1) = 0,  \quad 0 \leq k \leq n-1. 
\end{equation}

Now by considering  the transformation $\zeta = (\xt+i)/(\xt-i)$ we have 
\[
   I(\Muii) = \dashint_\mathcal{\T} \frac{\zeta}{\zeta-1} d \Muii(\zeta) 
   = \lim_{y \to \infty} \int_{-y}^{y} \frac{1}{2}(1 - i \xt) d \psi(\xt),
\]
where
\begin{equation} \label{Eq-RL-Measure}
     d \psi(\xt) = - d\Muii\big((\xt+i)/(\xt-i)\big). 
\end{equation}
Hence, by considering  the principal value integral on the real line defined by 
\begin{equation} \label{Eq-PVI-RealLine}
   \dashint_{-\infty}^{\infty} f(\xt) d\psi(\xt) = 
\lim_{y \to \infty} \int_{-y}^{y} f(\xt) d \psi(\xt),
\end{equation}
if  $I(\Muii)$ exists then 
\[
   \dashint_{-\infty}^{\infty} \xt d \psi(\xt) = -2 \Im(I(\Muii)). 
\]
Consequently, from \eqref{Eq-Lorthogonality-2},  
\begin{equation} \label{Eq-Ortogonality-P-3}
      \dashint_{-\infty}^{\infty} r_{k}^{(n)}(\xt) \frac{P_{n}(\xt)}{(\xt^2+1)^{n}} \,(\xt^2+1)  d \psi(\xt) + \big[c_{1} - \dashint_{-\infty}^{\infty} \xt\, d\psi(\xt)\big] \mathfrak{p}_{n} = 0, \quad 0 \leq k \leq n-1. 
\end{equation}
where $r_{k}^{(n)}(\xt) = (\xt+i)^{k} (\xt-i)^{n-1-k}$ and $\mathfrak{p}_n = (1-\ell_{1}) \cdots(1-\ell_{n})$ is the leading  coefficient of $P_{n}$.

\begin{theo} \label{Thm-PV-Orthonality-M}  Given the $R_{II}$ type recurrence, let  $\Muii$ be the probability measure on the unit circle  obtained as in Theorem \ref{Thm-Orthogonality1-UC}.  If the measure $\Muii$   is such that the principal value integral $ I(\Muii)$ defined by \eqref{Eq-PVI-UC} exists then with $\psi(\xt)$ given by \eqref{Eq-RL-Measure}  and with the principal value integral on the real line defined by \eqref{Eq-PVI-RealLine}, 
\[
      \dashint_{-\infty}^{\infty} \xt^{k} \, \frac{P_{n}(\xt)}{(\xt^2+1)^{n}} \, (\xt^2+1)  d \psi(\xt) = - \Big[c_{1} - \dashint_{-\infty}^{\infty} \xt\, d\psi(\xt)\Big] \mathfrak{p}_{n}\, \delta_{n-1,k}, \quad 0 \leq k \leq n-1,
\]
for $n \geq 1$.
\end{theo}

\begin{proof}
Since the set of polynomials $\{r_j^{(n)}\}_{j=0}^{n-1}$ is linearly independent in $\mathbb{P}_{n-1}$ we can write 
\[
     \xt^{k} = \sum_{j=0}^{n-1}q_{k,j}^{(n)}\,r_j^{(n)}(\xt),  \quad 0 \leq k \leq n-1.   
 \]
 Hence, from \eqref{Eq-Ortogonality-P-3} 
\[
      \dashint_{-\infty}^{\infty} \xt^{k} \, \frac{P_{n}(\xt)}{(\xt^2+1)^{n}} \, (\xt^2+1)  d \psi(\xt) + \Big[c_{1} - \dashint_{-\infty}^{\infty} \xt\, d\psi(\xt)\Big] \mathfrak{p}_{n}\, \sum_{j=1}^{n-1}q_{k,j}^{(n)} = 0 , \quad 0 \leq k \leq n-1.
\]
Now considering the limit of $\xt^{-n+1} \sum_{j=0}^{n-1}q_{k,j}^{(n)}r_j^{(n)}(\xt)$ as $\xt \to \infty$ we find that $\sum_{j=0}^{n-1}q_{k,j}^{(n)} = 0$ for $k=0,1, \ldots, n-2$ and $\sum_{j=0}^{n-1}q_{n-1,j}^{(n)} = 1$, which gives the required result in the theorem. 
\end{proof}

We are now able to consider the proof of Theorem \ref{Thm-Special-R2Type-MF}. 

\begin{proof}[\bf Proof of Theorem \ref{Thm-Special-R2Type-MF}]  

Since the positive chain sequence $\{d_{n+1}\}_{n \geq 1}$ has multiple parameter sequences, from Theorem \ref{Thm-Orthogonality1-UC} the associated measure is such that the integral $\int_{\T} |\zeta-1|^{-2} d \Muii(\zeta)$ exists.  Moreover, from the remarks that follow Theorem \ref{Thm-BRS2016-2} the polynomials $R_{n}$ given by \eqref{Eq-TTRR-R} satisfy $R_n(z) = R_n(0; z)$, $n \geq 0$  and 
\begin{equation} \label{Eq-Ortho3-Rn}
    \int_{\mathbb{T}} \zeta^{-n+k}R_{n}(\zeta)  (1-\zeta) d \nu(\zeta)  = \mathcal{N}^{(0)}[z^{-n+k}R_n(0;z)]  = 0,  \quad 0 \leq k \leq n-1, 
\end{equation}
where the probability measure on the unit circle $\nu$ is given  by 
\[
    \const \, d \nu(\zeta) = \frac{\zeta}{(\zeta-1)(1-\zeta)} d \Muii(\zeta) = \frac{1}{|\zeta-1|^{2}} d \Muii(\zeta).   
\]
Now from results obtained in \cite{Costa-Felix-Ranga-JAT2013} and from some further observations made later (see for example \cite{BracRangaSwami-2016}), if $\{\beta_{n}\}_{n \geq 0}$ is the sequence of Verblunsky coefficients with respect to the measure $\nu$ then for the coefficients of  the recurrence relation \eqref{Eq-TTRR-R} of $R_{n}$  we also have 
\begin{equation}\label{Eq-CoeffsTTRR-2}
  \begin{array}{l}
    \dsp c_{n} = \frac{-\mathcal{I}m \big(\tau_{n-1}\beta_{n-1}\big)} {1-\mathcal{R}e\big(\tau_{n-1}\beta_{n-1}\big)}   \quad \mbox{and} \quad d_{n+1} =  \big(1-g_{n}\big)g_{n+1}, \quad n \geq 1
  \end{array} 
\end{equation}
and 
\[
    \frac{1}{\prod_{j=1}^{n}(1-g_{j})}R_n(z) = K_n(z, 1),
\]
where $K_n(z, w)$ represents the Christoffel-Darboux kernels with respect to the probability measure $\nu$. Here, $\tau_0 =1$,  
\[
     g_{n} = \frac{1}{2} \frac{\big|1 - \tau_{n-1} \beta_{n-1}\big|^2}{\big[1 - \mathcal{R}e\big(\tau_{n-1}\beta_{n-1}\big)\big]} \quad \mbox{and} \quad    \tau_{n}   = \frac{\tau_{n-1} - \overline{\beta}_{n-1}}{1 - \tau_{n-1} \beta_{n-1}}, \quad n \geq 1.
\]
Since the measure $\nu$ does not have a pure point at $z=1$, the above parameter sequence $\{g_{n+1}\}_{n \geq 0}$ of the positive chain sequence  $\{d_{n+1}\}_{n \geq 1}$ is also its  maximal parameter sequence $\{M_{n+1}\}_{n \geq 0}$.  

The following reciprocal formulas can also be easily verified.
\[
     \beta_{n-1} = \frac{1}{\tau_{n-1}}\,\frac{1-2M_{n} - i c_n}{1 - i c_n} \quad \mbox{and} \quad \tau_{n} =  \tau_{n-1}\frac{1-ic_n}{1+ic_n} 
\]
for $n \geq 1$, with $\tau_{0}= 1$.
 
Now we consider the transformation $\zeta = (\xt+i)/(\xt-i)$.  Then we have from \eqref{Eq-P-to-R} and \eqref{Eq-Ortho3-Rn}  
\[
       \int_{-\infty}^{\infty} \xt^{k} \, \frac{P_{n}(\xt)}{(\xt^2+1)^{n}} \, d \varphi(\xt)  = 0, \quad k=0,1, \ldots, n-1.
\]
where $d \varphi(\xt) = - d \nu((\xt+i)/(\xt-i))$. 

To obtain the value of $\gamma_{n} = \int_{-\infty}^{\infty}  \frac{\xt^{n}P_{n}(\xt)}{(\xt^2+1)^{n}} \, d \varphi(\xt)$, first we observe from the $R_{II}$ type recurrence \eqref{Eq-Special-R2Type-RR},
\[
     (\xt^2+1) \frac{P_{n+1}(\xt)}{(\xt^2+1)^{n+1}} = (\xt - c_{n+1}) \frac{P_{n}(\xt)}{(\xt^2+1)^{n}}  - d_{n+1} \frac{P_{n-1}(\xt)}{(\xt^2+1)^{n-1}}, \quad n \geq 1.
\]
Hence, $ \gamma_{n+1} = \gamma_{n} - d_{n+1} \gamma_{n-1}$, which is equivalent to 
\[
   d_{n+1} = \frac{\gamma_{n}}{\gamma_{n-1}}\big[1 - \frac{\gamma_{n+1}}{\gamma_{n}}  \big], \quad n \geq 1.
\]
Clarely, $\gamma_{0} = \int_{-\infty}^{\infty} d \varphi(\xt) = \int_{\T} d \nu(\zeta) = 1$. Thus, all we have to show is  $\gamma_{1} =  (1-M_1)$.  

With $\int_{-\infty}^{\infty}  \frac{P_{1}(\xt)}{(\xt^2+1)} \, d \varphi(\xt) = 0$ we can write 
\[
    \gamma_{1} = \int_{-\infty}^{\infty} (\xt+c_1) \frac{P_{1}(\xt)}{\xt^2+1} d \varphi(\xt) = 1 - (1+c_1^{2}) \int_{-\infty}^{\infty} \frac{1}{\xt^2+1} d \varphi(\xt).
\]
However,
\[
    \int_{-\infty}^{\infty} \frac{1}{\xt^2+1} d \varphi(\xt) = \frac{1}{4}\int_{\T} \zeta^{-1}(\zeta-1)^2 d \nu(\zeta) = \frac{1}{2} [1 - \Re(\tilde{\alpha}_0)].   
\]
Hence, together with the result for $c_1$ in \eqref{Eq-CoeffsTTRR-2} we find $\gamma_{1} =  (1-M_1)$. 
\end{proof}

\setcounter{equation}{0}
\section{Examples } \label{Sec-Examples}

We justify the statements made  in Theorems \ref{Thm-Special-R2Type-MF} and \ref{Thm-Orthogonality1-UC}, and also the results given in Section \ref{Sec-OPUC-to-POPUC2} with the following examples.  \\[-0.5ex]

\noindent{\bf Example 1}. We consider the $R_{II}$ type recurrence \eqref{Eq-Special-R2Type-RR} with $c_n = 0$, $n \geq 1$, $d_{2} = 1/2$ and $d_{n+1} = 1/4$, $n \geq 2$. The sequence $\{d_{n+1}\}_{n\geq 1}$ is a single parameter positive chain sequence with its minimal parameter sequence $\{\ell_{n+1}\}_{n \geq 0} = \{0, 1/2, 1/2, 1/2, \ldots\}$.

From the theory of  difference equations it is easily verified that 
\[
       P_{n}(\xt) = \big(\frac{\xt-i}{2}\big)^n + \big(\frac{\xt+i}{2}\big)^n, \quad n \geq 1. 
\]
Hence, from \eqref{Eq-Rationals-1} we have 
\[
     \Fi_{0}(\xt) = 1, \quad \mbox{and} \quad \Fi_{n}(\xt)  = \frac{(-1)^n}{\sqrt{2}}\big[1 + \frac{(\xt+i)^n}{(\xt-i)^n}\big], \ \ n \geq 1. 
\]
From this we successively find from the results given in Section \ref{Sec-EVP-to-Orthogonality} that 
\[
     \hat{\Fi}_{n}(\xt) = \frac{(-1)^n}{2} \Big[\frac{(\xt+i)^{n}}{(\xt-i)^{n}} - \frac{(\xt+i)^{n-1}}{(\xt-i)^{n-1}} \Big], \ \ \ n \geq 1,
\]
\[
    R_{0}(z) = 1, \quad R_{n}(z) = z^n +1, \ \ n \geq 1, 
\]
\[
    \hat{R}_{n}(z) = z^{n} - z^{n-1}, \ \ \  n \geq 1 
\]
and 
\[
    \Phi_{n}(z) = z^{n}, \quad n \geq 0.
\]
Thus, clearly the probability measure $\mu$ that follows from Theorem \ref{Thm-Orthogonality1-UC} is the Lebesgue measure. That is, $d \mu(\zeta) = \frac{1}{2\pi i\zeta} d \zeta$ and all the associated Verblunsky coefficients are equal to $0$. 

For the Lebesgue  measure the integral $\int_{\T}|\zeta-1|^{-2} d \mu(\zeta)$ does not exist. However, we observe that the principal value integral $I(\mu)$ exists.  To be precise, we have
\[
   I(\mu) = \frac{1}{2\pi} \lim_{\epsilon \to 0} \Big[ \int_{0+\epsilon}^{\pi} \frac{e^{i\theta}}{e^{i\theta}-1} d \theta + \int_{0+\epsilon}^{\pi} \frac{e^{-i\theta}}{e^{-i\theta}-1} d \theta\Big] = \frac{1}{2\pi} \lim_{\epsilon \to 0} \int_{0+\epsilon}^{\pi} d\theta,
\]
from which $I(\mu) = 1/2$. Hence, from Theorem \ref{Thm-BRS2016-2} (see also \cite{BracRangaSwami-2016}) 
\[
     \tau_{n+1}(s) = \frac{1 +i 2s}{1-i2s} , \quad n \geq 0   
\]
and for the polynomials 
\[
   \quad R_n(s;z)  \frac{1+\tau_1(s)}{2} \prod_{k=1}^{n-1} \frac{ 1 + \Re[\tau_{k}(s)\alpha_{k-1}]}{1 + \tau_{k}(s) \alpha_{k-1}} = z\Phi_{n-1}(z) + \tau_{n}(s)\Phi_{n-1}^{\ast}(z), \quad n \geq 1,
\]
there hold  $R_{1}(s;z) = (1+ic_{1}(s))\zeta + (1-i c_{1}(s))$ and 
\[
    R_{k+1}(s;z) = [(1+ic_{k+1}(s))z + (1-ic_{k+1}(s))] R_{k}(z) - 4 d_{k+1}(s) z R_{k-1}(z), \quad k \geq 1, 
\]
where
\begin{equation*} \label{Eq-Example1}
      c_1(s) = -2s, \quad c_{n+1}(s) = 0,  \   n \geq 1, \qquad d_2(s) = d_2 = \frac{1}{2}, \quad d_{n+2}(s) = d_{n+1} = \frac{1}{4}, \  n  \geq 1.
\end{equation*} 
Thus, we conclude that with a choice of $s$ such that $-\infty < s < \infty$ the real sequence $\{c_n(s)\}_{n \geq 1} = \{-2s, 0, 0, 0, \ldots\}$ and and positive chain sequence $\{d_{n+1}\}_{n \geq 1} = \linebreak \{1/2, 1/4, 1/4, 1/4, \ldots\}$ used  together with Theorem \ref{Thm-Orthogonality1-UC} leads to the same measure which is the Lebesgue measure.  

If the sequence of polynomials $\{P_{n}(s;\xt)\}$ are  given by $P_{0}(s; \xt) = 1$, $P_{1}(s; \xt) = \xt - c_1(s)$ and 
\[
     P_{n+1}(s; \xt) = [\xt - c_{n+1}(s)] P_{n}(s;\xt)  - d_{n+1} (\xt^2 + 1) P_{n-1}(s;\xt), \quad n \geq 1.
\]
then from Theorem \ref{Thm-PV-Orthonality-M} we have 
\[
     d \psi(\xt) = \frac{1}{\pi} \frac{1}{\xt^2+1} d\xt, \qquad \dashint_{-\infty}^{\infty} \xt\, d \psi(\xt)=  \lim_{y \to \infty} \Big[\frac{1}{2\pi}\ln (\xt^2+1)\Big]_{-y}^{y} = 0 
\]
and
\[
      \frac{1}{\pi} \dashint_{-\infty}^{\infty} \xt^{k} \, \frac{P_{n}(s;\xt)}{(\xt^2+1)^{n}} \,   d \xt  =  - c_{1}(s) \mathfrak{p}_{n} \, \delta_{n-1, k}, \quad 0 \leq k \leq n-1,
\] \\[0ex]
where $\mathfrak{p}_n = (1-\ell_{1}) \cdots(1-\ell_{n}) = 1/2^{n-1}$. 

\vspace{4ex}

\noindent {\bf Example 2}.  With  $s$ such that $-\infty < s < \infty$, we consider the $R_{II}$ type recurrence \eqref{Eq-Special-R2Type-RR}, where  $c_n = c_n(s)$ and $d_{n+1} = d_{n+1}(s) = [1- \ell_{n}(s)] \ell_{n+1}(s)$, for $n \geq 1$, are given by 
 $c_{1}(s) = \kappa - 2s$, $\ell_{1}(s) =0$,    
\begin{equation*}   \label{Eq-c-psi1-2}
     \begin{array}{ll}
         \displaystyle c_{4m+2}(s) &\!\! = \displaystyle \frac{-\kappa}{1+4m\kappa}\,
\frac{(1-\kappa)^2[1+(4m+1)\kappa] + 4s\kappa(1-\kappa) - 4s^2 [1+(4m-1)\kappa]}
{(1-\kappa)^2[1+(4m+1)\kappa]+4s^2 [1+(4m-1)\kappa]},   
                     \\[3ex] 
         \displaystyle c_{4m+3}(s) &\!\! = \displaystyle \frac{-4s\kappa}{1+(4m+1)\kappa}\, 
\frac{(1-\kappa)[1+(4m+1)\kappa] + 2 s \kappa }
{(1-\kappa)^2[1+(4m+1)\kappa] + 4s\kappa(1-\kappa) + 4s^2[1+(4m+1)\kappa]},  
                     \\[3ex] 
         \displaystyle c_{4m+4}(s) &\!\! = \displaystyle \frac{\kappa}{1+(4m+2)\kappa}\, 
\frac{(1-\kappa)^2[1+(4m+1)\kappa] + 4s\kappa(1-\kappa)   - 4s^2 [1+(4m+3)\kappa]}
{(1-\kappa)^2[1+(4m+1)\kappa]+4s^2 [1+(4m+3)\kappa]},  
                    \\[3ex] 
         \displaystyle c_{4m+5}(s)  &\!\! = \displaystyle  \frac{-2\kappa (1-\kappa)}{1+(4m+3)\kappa}\, 
\frac{(1-\kappa) \kappa - 2 s [1+(4m+3)\kappa]}
{(1-\kappa)^2[1+(4m+3)\kappa] - 4s\kappa(1-\kappa) + 4s^2[1+(4m+3)\kappa]},
       \end{array}  
\end{equation*}
\[
    \begin{array}{l}
          \ell_{4m+2}(s)   = \dsp  \frac{1+(4m-1)\kappa}{2[1+(4m)\kappa]^2} \times \\[2ex]
   \dsp \quad  \frac{(1-\kappa)^2[1+(4m+1)\kappa]^2 +4s\kappa(1-\kappa)[1+(4m+1)\kappa] +4s^2\big[\kappa^2 + [1+(4m)\kappa]^2\big]}
{(1-\kappa)^2[1+(4m+1)\kappa]    + 4s^2[1+(4m-1)\kappa]},      
    \end{array}  
\]
\[
   \begin{array}{l} 
          \ell_{4m+3}(s)  =  \displaystyle \frac{1+(4m)\kappa}{2[1+(4m+1)\kappa]^2} \times \\[2ex]
  \displaystyle  \quad \frac{(1-\kappa)^2[1+(4m+1)\kappa]^2 +4s\kappa(1-\kappa)[1+(4m+1)\kappa] +4s^2\big[\kappa^2 + [1+(4m+2)\kappa]^2\big]}
{(1-\kappa)^2[1+(4m+1)\kappa] + 4s\kappa (1-\kappa) + 4s^2[1+(4m+1)\kappa]}, 
   \end{array}  
\]
\[
   \begin{array}{l}
      \ell_{4m+4}(s)  =  \displaystyle \frac{1+(4m+1)\kappa}{2[1+(4m+2)\kappa]^2} \times \\[2ex]
   \displaystyle  \quad \frac{(1-\kappa)^2\big[\kappa^2 + [1+(4m+2)\kappa]^2\big] -4s\kappa(1-\kappa)[1+(4m+3)\kappa] + 4s^2[1+(4m+3)\kappa]^2}
{(1-\kappa)^2[1+(4m+1)\kappa] + 4s^2[1+(4m+3)\kappa]} ,                
   \end{array}  
\]
and
\[
   \begin{array}{l}
         \ell_{4m+5}(s)   =  \displaystyle  \frac{1+(4m+2)\kappa}{2[1+(4m+3)\kappa]^2} \times \\[2ex]
  \displaystyle  \quad \frac{(1-\kappa)^2\big[\kappa^2 + [1+(4m+4)\kappa]^2\big] -4s\kappa(1-\kappa)[1+(4m+3)\kappa] +4s^2[1+(4m+3)\kappa]^2}
{(1-\kappa)^2[1+(4m+3)\kappa] - 4s\kappa (1-\kappa)  + 4s^2[1+(4m+3)\kappa]} ,
   \end{array}  
\]
for $m \geq 0$. Here, we take $\kappa$ to be such that  $0 < \kappa < 1$. We also denote by $P_{n}(s;x)$ the polynomials obtain from \eqref{Eq-Special-R2Type-RR}. 

Independent of the value of $s$,  the measure that follows from Theorem \ref{Thm-Orthogonality1-UC} is (see \cite{BracSilvaSR-AMC2015})  
\[
    d\mu(\zeta) = (1 - \kappa) \frac{1}{2\pi i \zeta} d \zeta + \kappa\, \delta_{i}. 
\]
The above probability measure $\mu$ is a modification of the Lebesgue measure so that it has  an isolated mass (pure point) of size $\kappa$ at the point $z = i$. 

For the principal value integral 
\[
   I(\mu) = \dashint_\mathcal{\T} \frac{\zeta}{\zeta-1} d \Muii(\zeta) = \frac{1-\kappa}{2\pi} \lim_{\epsilon \to 0} \Big[ \int_{0+\epsilon}^{2\pi-\epsilon} \frac{e^{i\theta}}{e^{i\theta}-1} d \theta\Big] + \frac{\kappa\, i}{i - 1},
\]
we obtain  $I(\mu) = (1 - i \kappa)/2$. Hence, by using the results given by Theorem  \ref{Thm-PV-Orthonality-M} we have  
\[
     d \psi(\xt) = \frac{1}{\pi} \frac{1-\kappa}{\xt^2+1} d\xt + \kappa\, \delta_{1}, \qquad \dashint_{-\infty}^{\infty} \xt\, d \psi(\xt)=  \kappa 
\]
and
\[
      \frac{1}{\pi} \dashint_{-\infty}^{\infty} \xt^{k} \, \frac{P_{n}(s;\xt)}{(\xt^2+1)^{n}} \, (x^2+1)  d \psi(\xt)  =  - [c_{1}(s) - \kappa] \mathfrak{p}_{n}(s) \, \delta_{n-1, k}, \quad 0 \leq k \leq n-1,
\]
where $\mathfrak{p}_n(s) = (1-\ell_{1}(s)) \cdots(1-\ell_{n}(s))$.  Here, $\delta_{1}$ represents the Dirac measure concentrated at the point $x=1$. 

\vspace{4ex}

\noindent {\bf Example 3}.  We now consider the $R_{II}$ type recurrence \eqref{Eq-Special-R2Type-RR} with 
\begin{equation} \label{Eq-Example3}
   c_n = 0 \quad \mbox{and} \quad d_{n+1} = 1/4, \quad n \geq 1. 
\end{equation}
The sequence $\{d_{n+1}\}_{n\geq 1}$ is not a single parameter positive chain sequence. Its   minimal parameter sequence is $\{\ell_{n+1}\}_{n \geq 0}$, where $\ell_{n+1} = n/(2(n+1))$, $n \geq 0$ and its maximal parameter sequence  $\{M_{n+1}\}_{n \geq 0}$ is such that  $M_{n+1} = 1/2$, $n \geq 0$.

Again from  the theory of  difference equations it is easily verified that 
\[
       P_{n}(\xt) = i \big(\frac{\xt-i}{2}\big)^{n+1}  - i \big(\frac{\xt+i}{2}\big)^{n+1}, \quad n \geq 0. 
\]
Hence, from \eqref{Eq-Rationals-1} we have 
\[
      \Fi_{n}(\xt)  = (-1)^n\frac{1}{2}i(\xt-i)\big[1 - \frac{(\xt+i)^{n+1}}{(\xt-i)^{n+1}}\big], \ \ n \geq 0. 
\]
From this we successively find from the results given in Section \ref{Sec-EVP-to-Orthogonality} that 
\[
     \hat{\Fi}_{n}(\xt) = (-1)^n\frac{\sqrt{\ell_{n+1}}}{2} i (\xt-i)  \Big[1 - \frac{n+1}{n}- \frac{(\xt+i)^{n+1}}{(\xt-i)^{n+1}} + \frac{n+1}{n} \frac{(\xt+i)^{n}}{(\xt-i)^{n}} \Big], \ \ \ n \geq 1,
\]
\[
     R_{n}(z) = \frac{z^{n+1} - 1}{z - 1}, \ \ n \geq 0, 
\]
\[
    \hat{R}_{n}(z) = \frac{1}{n(z-1)}\big[n (z^{n+1}-1) - (n+1) (z^{n} -1 )  \big], \ \ \  n \geq 1 
\]
and 
\[
    \Phi_{n-1}(z) = \frac{1}{n(z-1)^2}\big[n (z^{n+1}-1) - (n+1) (z^{n} -1 )  \big], \ \ \  n \geq 1.
\]

The corresponding probability measure $\mu$ is found to be  
\[
     d \mu(\zeta) =  \frac{(1-\zeta)(\zeta -1)}{4\pi i \zeta^2} d \zeta.
\]
Now we  derive the polynomials $R_n(s;z)$ and their orthogonality property given by Theorem \ref{Thm-BRS2016-2}.  Again, many information regarding the orthogonal polynomials on the unit circle  associated with this measure are well known. 

For example, the associated Verblunsky coefficients are
\[
      \alpha_{n-1} = - \frac{1}{n+1}, \quad n \geq 1. 
\]
Clearly the integral $\int_{\T}|\zeta - 1|^{-2} d \mu(\zeta)$ exists and that one can also easily verify that 
\[
       I(\mu) = \int_{\mathcal{C}} \frac{\zeta}{\zeta - 1} d \mu(\zeta)  = \frac{1}{2}.
\]
We also easily verify by induction that the sequence $\{\tau_{n}(s)\}$ generated by 
\[
	     \dsp \tau_{1}(s)=  \frac{I+is}{\overline{I} -is}\quad \mbox{and} \quad  \tau_{n+1}(s) = \frac{\tau_{n}(s)  + \overline{\alpha_{n-1}}}{1 + \tau_{n}(s) \alpha_{n-1}}, \quad n \geq 1,
\]
can be explicitly given by (see \cite{BracRangaSwami-2016})
\[
   \tau_{n+1}(s)  =  \frac{1 + i (n+1)(n+2)s}{1 - i (n+1)(n+2)s}, \quad n\geq 0. 
\]
With this we obtain  
\[
     c_n(s) = - \frac{2ns}{1 + (n^2-1)(n+1)^2 s^2} \quad   \mbox{and}  \quad d_{n+1}(s) = [1-\ell_{n}(s)]\ell_{n+1}(s), 
\]
for $n \geq 1$, where 
\[
   \ell_{n+1}(s) = \frac{n}{2(n+1)} \frac{1 + (n+1)^2 (n+2)^2 s^2}{1 + n (n+1)^2 (n+2)s^2}, \quad n \geq 0.
\]

Since, 
\[
      1 - \ell_{n+1}(s) = \frac{n+2}{2(n+1)} \frac{1 + n^2 (n+1)^2 s^2}{1 + n (n+1)^2 (n+2)s^2},  \quad n \geq 1, 
\]
we have 
\[
     \sum_{n =1}^{\infty}  \prod_{k=1}^{n} \frac{\ell_{k+1}(s)}{1 - \ell_{k+1}(s)} =  \frac{2}{1 + 4s^2}\sum_{n =1}^{\infty}  \Big[\frac{1}{(n+1)(n+2)} + (n+1)(n+2) s^2 \Big].
\]
The above series converges if $s = 0$ and diverges otherwise.  Hence, by Wall's criteria the positive chain sequence $\{d_{n+1}(0)\}$ is a multiple parameter positive chain sequence and $\{d_{n+1}(s)\}$ for $s \neq 0$ is a single parameter positive chain sequence. 

Now if the sequence of polynomials $\{P_{n}(s;\xt)\}$ are  given by $P_{0}(s; \xt) = 1$, $P_{1}(s; \xt) = \xt - c_1(s)$ and 
\[
     P_{n+1}(s; \xt) = [\xt - c_{n+1}(s)] P_{n}(s;\xt)  - d_{n+1}(s)  (\xt^2 + 1) P_{n-1}(s;\xt), \quad n \geq 1.
\]
then from Theorem \ref{Thm-PV-Orthonality-M}  
\[
     d \psi(\xt) = \frac{2}{\pi} \frac{1}{(\xt^2+1)^2} d\xt, \qquad \int_{-\infty}^{\infty} \xt\, d \psi(\xt) = 0 
\]
and
\[
      \frac{2}{\pi} \int_{-\infty}^{\infty} \xt^{k} \, \frac{P_{n}(s;\xt)}{(\xt^2+1)^{n}} \, \frac{1}{\xt^2+1}  d \xt \ =  - \ c_{1}(s) \mathfrak{p}_{n}(s)\, \delta_{n-1,k}, \quad 0 \leq k \leq n-1,
\] \\[0ex] 
where $\mathfrak{p}_n(s) = (1-\ell_{1}(s)) \cdots(1-\ell_{n}(s))$.

Since $\{d_{n+1}(0)\}_{n \geq 1} = \{1/4\}_{n \geq 1}$ is a multiple parameter positive chain sequence, from Theorem \ref{Thm-Special-R2Type-MF} we have for the polynomials $\{P_n\} = \{P_n(0;.)\}$ given by \eqref{Eq-Special-R2Type-RR} and \eqref{Eq-Example3} 
\[
      \frac{2}{\pi} \int_{-\infty}^{\infty} \xt^{k} \, \frac{P_{n}(\xt)}{(\xt^2+1)^{n}} \, \frac{1}{\xt^2+1}  d \xt = \frac{1}{2^n} \delta_{n,k}, \quad k=0,1, \ldots, n.
\]

\vspace{4ex}

\noindent {\bf Example 4}. Here we consider the $R_{II}$ type recurrence \eqref{Eq-Special-R2Type-RR} with  $\{c_n\}_{n\geq 1}$ and $\{d_{n+1}\}_{n\geq 1}$ given by
\begin{equation} \label{Eq-Example4}
 \begin{array}l
  \dsp c_n = \frac{\eta}{\lambda+n}, \quad d_{n+1} = \frac{1}{4} \frac{n(2\lambda+n+1)}{(\lambda+n)(\lambda+n+1)}, \ \ n\geq 1,
 \end{array}
\end{equation}
where \ $\lambda, \eta \in \mathbb{R}$  and  $\lambda > -1$.

Observe that for the minimal parameter sequence of the positive chain sequence $\{d_{n+1}\}_{n\geq 1}$ we have 
\begin{equation*} \label{Eq-Special-ParamSeq}
   \quad \ell_{n+1} = \frac{n}{2(\lambda+n+1)}, \quad \ n \geq 0.
\end{equation*}
The sequence $\{d_{n+1}\}_{n=1}^{\infty}$ is a single parameter chain sequence for $-1/2 \geq \lambda > -1$ and it is a multiple parameter sequence otherwise. When  $\lambda > -1/2$ the maximal parameter sequence $\{M_{n+1}\}_{n=0}^{\infty}$ of $\{d_{n+1}\}_{n=1}^{\infty}$ is such that 
\[
      M_{n+1} = \frac{1}{2}\frac{2\lambda +n +1}{\lambda+n+1},  \quad n \geq 0.
\]

The polynomials $R_n$ obtained from the above sequences $\{c_n\}$ and $\{d_{n+1}\}$, together with the recurrence formula \eqref{Eq-TTRR-R} have already been studied in \cite{BracRangaSwami-2016}, and we have 
\begin{equation*} \label{Eq-HypergeometricRn}
      R_n(z) = \frac{(2\lambda+2)_n}{(\lambda+1)_n} \, _2F_1(-n,b+1;\,b+\overline{b}+2;\,1-z),  \quad n \geq 1,
\end{equation*}
where  $b = \lambda + i \eta$.  Here, $ _2F_1(-n,b;\,c;\,z)$ represents a Gaussian hypergeometric polynomial of degree $n$ in $z$. More about the properties of such polynomials and about general Gaussian hypergeometric functions $ _2F_1(a,b;\,c;\,z)$, we cite \cite{Andrews-Book}.

From  \eqref{Eq-P-to-R} we then have 
\[
     P_{n}(\xt) = \frac{(2\lambda+2)_n}{(\lambda+1)_n} \frac{(\xt-i)^n}{2^n}\, _2F_1(-n,b+1;\,b+\overline{b}+2;\,-2i/(\xt-i)), \quad n \geq 1. 
\]

The nontrivial probability measure $\mu$ that follows from  Theorem \ref{Thm-Orthogonality1-UC} is found to be 
\[
    d \mu(e^{i\theta})= \frac{2^{b+\overline{b}+2}|\Gamma(b+2)|^2 }{2\pi\, \Gamma(b+\overline{b}+3)} e^{(\pi-\theta)\Im(b)} [\sin^{2}(\theta/2)]^{\Re(b)+1} d\theta . \qquad
\]
The associated monic orthogonal polynomials on the unit circle and Verblunsky coefficients  are (see \cite{SR-PAMS2010})
\[
    \begin{array}l 
          \dsp \Phi_n(z) = \frac{(b+\overline{b}+3)_n}{(b+2)_n} \, _2F_1(-n,b+2;\,b+\overline{b}+3;\,1-z),  \\[3ex]
          \dsp \alpha_{n-1} = - \frac{(b+1)_{n}}{(\overline{b}+2)_{n}},
    \end{array}   n \geq 1.
\]
When $\lambda > -1/2$ the sequence $\{d_{n+1}\}_{n\geq 1}$ is a multiple parameter positive chain sequence and its maximal parameter sequence $\{M_{n+1}\}_{n\geq 0}$ is such that
\[
     M_{n+1} = \frac{1}{2}\, \frac{2\lambda+n +1}{\lambda+n+1}, \quad n \geq 0.
\]

Hence, for $\lambda > -1/2$ we can consider the probability measure 
\[
     d \nu(e^{i\theta}) = \const  \frac{1}{|\zeta-1|^2} d \mu(e^{i\theta}) = \frac{2^{b+\overline{b}}|\Gamma(b+1)|^2 }{2\pi\, \Gamma(b+\overline{b}+1)} e^{(\pi-\theta)\Im(b)} [\sin^{2}(\theta/2)]^{\Re(b)} d\theta,  
\]
and from Theorem \ref{Thm-Special-R2Type-MF} the polynomials $P_{n}$ given by \eqref{Eq-Special-R2Type-RR} and \eqref{Eq-Example4}  satisfy the orthogonality   
\[
      \int_{-\infty}^{\infty} \xt^{j} \, \frac{P_{n}(\xt)}{(\xt^2+1)^{n}} \,  d \varphi(\xt) = \gamma_n \delta_{n,k}, \quad j=0,1, \ldots, n,
\]
where 
\[
      d \varphi(\xt) = - d \nu\big(\frac{x+i}{x-i}\big) = \frac{2^{b+\overline{b}+1}|\Gamma(b+1)|^2 }{2\pi\, \Gamma(b+\overline{b}+1)}\,  e^{[\pi -2\,arccot(\xt)]\eta} (\xt^2 + 1)^{-\lambda-1} d\xt,
\]
$\gamma_{0} = \int_{-\infty}^{\infty} d\varphi(\xt) = \int_{\T} d \nu(\zeta) = 1$ and $\gamma_{n} = (1- M_{n})\gamma_{n-1}$, $n \geq 1$. 

\vspace{2ex}


\end{document}